\documentclass[final,leqno,onefignum]{siamart}
\usepackage{epsfig,amsfonts,latexsym,amsmath,amssymb}
\usepackage{graphicx}
\usepackage{subfig}
\usepackage{hyperref}
\usepackage{booktabs}
\usepackage{algorithm}
\usepackage[noend]{algorithmic}

\usepackage{enumitem}

\renewcommand{\tablename}{Fig.}
\makeatletter
\renewcommand*{\thetable}{\arabic{table}}
\renewcommand*{\thefigure}{\arabic{figure}}
\let\c@table\c@figure
\makeatother

\title{An Exact Redatuming Procedure for the Inverse Boundary Value Problem for the Wave Equation}

\author{Maarten V. de  Hoop\ \footnotemark[2]
\and Paul Kepley\ \footnotemark[3]
\and Lauri Oksanen \footnotemark[4]}

\newcounter{notecounter}
\newcommand{\note}[1]{
  \stepcounter{notecounter}
  \marginpar{\footnotesize Note \thenotecounter: \newline #1}
}

\newcommand{\norm}[1]{\left\|#1 \right\|}

\def\R{\mathbb{R}}
\def\C{\mathbb{C}}
\def\N{\mathbb{N}}

\def\p{\partial}

\def\curlyL{\mathcal{L}}

\def\curlyK{\mathcal{K}}

\def\Wint{W_{\text{int}}}
\def\supp{\text{supp}}

\DeclareMathOperator*{\argmin}{argmin}

\DeclareMathOperator{\linearSpan}{span}
\DeclareMathOperator\SecondFundamentalForm{II}

\def\LaplaceBeltrami{\Delta_{g}}
\def\pND{\p_{\nu}}
\def\boundaryConds{\pND}

\def\dS{dS_g}
\def\dtdS{dt\otimes\dS}

\newenvironment{namedProblem}[1]{
  \begin{itemize}[leftmargin=0.75in, rightmargin=0.75in, topsep=6pt]
  \item[(#1)] \itshape
}{
  \end{itemize}
}

\begin{document}

\maketitle

%% Addresses:
\renewcommand{\thefootnote}{\fnsymbol{footnote}} 

\footnotetext[1]{An extended abstract covering portions of this paper
  appeared in the SEG Technical Program Expanded Abstracts 2016, doi:
  \url{http://dx.doi.org/10.1190/segam2016-13966222.1}.}

\footnotetext[2]{Simons Chair in Computational and Applied Mathematics
  and Earth Science, Rice University, Houston TX 77005, USA
  ({mdehoop@rice.edu}). The work of this author was supported in part
  by the Simons Foundation, NSF grant DMS-1559587, and members of the
  Geo-Mathematical Imaging Group at Rice University.}

\footnotetext[3]{Department of Mathematics, Purdue University, West
  Lafayette, IN 47907 ({pkepley@purdue.edu}). The work of this author
  was supported in part by the Geo-Mathematical Imaging Group at Rice
  University.}

\footnotetext[4]{Department of Mathematics, University College London,
  Gower Street, London WC1E 6BT, UK ({l.oksanen@ucl.ac.uk}).  The work
  of this author was partially supported by the Engineering and
  Physical Sciences Research Council (UK) grant EP/L026473/1.}
\renewcommand{\thefootnote}{\arabic{footnote}}

%% Set up the running details
\pagestyle{myheadings} 

\thispagestyle{plain} 

\markboth{M.V. DE HOOP, P. KEPLEY, AND L. OKSANEN}{EXACT REDATUMING FOR THE WAVE EQUATION}

\begin{abstract}
Redatuming is a data processing technique to
transform measurements recorded in one acquisition geometry
to an analogous data set corresponding to another acquisition geometry, for which there are no recorded measurements. 
We consider a redatuming problem for a wave equation on a bounded domain, or on a manifold with boundary, 
and  model data acquisition by a restriction of the associated Neumann-to-Dirichlet map.
This map models measurements with 
sources and receivers on an open subset $\Gamma$
contained in the boundary of the manifold. We model the wavespeed by a Riemannian metric, and suppose that the
metric is known in some coordinates in a neighborhood of $\Gamma$.
Our goal is to move sources and receivers into this known near boundary
region. We formulate redatuming as a collection of unique
continuation problems, and provide a two step procedure to solve the
redatuming problem. We investigate the stability of the first step
in this procedure, showing that it enjoys conditional H\"older
stability under suitable geometric hypotheses. In addition, we
provide computational experiments that demonstrate our redatuming
procedure.
\end{abstract}

\begin{keywords}
  Boundary Control method, redatuming,  wave equation
\end{keywords}

\begin{AMS}
  35R30, 35L05
\end{AMS}

\section{Introduction}

We consider an exact redatuming procedure for the inverse boundary
value problem for the wave equation.
We let $M$ be a bounded domain in $\R^n$, or more generally a smooth manifold with boundary, and assume that its boundary $\p M$ is smooth. Then, we consider
the wave equation
%\begin{equation}
\begin{align}
  \label{eqn:1stForwardProb}
  (\partial_t^2 - \LaplaceBeltrami)u(t,x) &= 0, & (t,x) \in (0, \infty) \times M,\\ 
  \boundaryConds u(t,x) &= f(t,x), & (t,x) \in (0,\infty) \times \p M,\notag\\
  u(0,x) = \partial_t u(0,x) &= 0, & x \in M, \notag
\end{align}
where $\Delta_g$ denotes the Laplace-Beltrami operator for a metric tensor $g$ on $M$. Let us remark that, in the case of a domain, this Riemannian formulation allows us to consider both the cases of isotropic and elliptically anisotropic wavespeeds.   We suppose that the metric $g$ is known, for some fixed $r > 0$
and in some fixed coordinates, in the domain of influence
$M(\Gamma,r)$, defined by:
\begin{equation}
    M(\Gamma,r) := \{ x \in M : d(x,\Gamma) \leq r\}.
\end{equation}
Outside of this set, the metric will be assumed to be unknown. 
We suppose that $g$ is smooth in $M(\Gamma,r)$, but allow for $g$ to possess singularities of
conormal type in the complement of this set.

%% \textbf{In addition, we can allow for $g$ to have conormal type
%%   singularities in $M(\Gamma,r)$ provided that an appropriate unique
%%   continuation principle holds. In particular, we could assume that
%%   either $(M,g)$ is a Riemannian polyhedron (see
%%   e.g. \cite{Kirpichnikova2012}) or that $g$ has a conformal jump
%%   discontinuity in $M(\Gamma,r)$ (see e.g. \cite{Oksanen2013b}).}

The term redatuming comes from the seismic literature, where it is
used to refer to procedures to synthesize
measurements for another set where data has not been recorded (see e.g. \cite{Mulder2005}). In the
present setting, we suppose that data has been collected on an open
subset $\Gamma \subset \p M$ in the form of the Neumann-to-Dirichlet
map (N-to-D map). Specifically, we suppose that for a fixed fixed time
$T > 0$, we have the N-to-D map $\Lambda_\Gamma^{2T}$, defined by:
\begin{equation*}
  \Lambda_\Gamma^{2T} f = u^f|_{(0,2T) \times \Gamma}, \quad f \in
  C_0^\infty((0,2T) \times \Gamma)
\end{equation*}
where $u^f$ is the solution of (\ref{eqn:1stForwardProb}).
Let $\Omega \subset M(\Gamma, r)$ be the set into which we would like to ``move'' the sources and receivers.
To make this precise, let $F$ be an interior source supported in
$[0,T/2] \times \Omega$, and let $w^F$ solve
\begin{align}
  \label{eqn:2ndForwardProb}
  (\partial_t^2 - \LaplaceBeltrami) w(t,x) &= F(t,x), & (t,x) \in (0, \infty) \times M,\\ 
  \boundaryConds w(t,x) &= 0, & (t,x) \in (0,\infty) \times \p M, \notag\\
  w(0,x) = \partial_t w(0,x) &= 0, & x \in M.\notag
%\end{array}
%\end{equation}
\end{align}
We define the map:
\begin{equation}
  %\label{eqn:DefineCurlyL}
  \curlyL : F \mapsto w^F|_{[0,T/2] \times \Omega}, \qquad \text{for $F \in L^2([0,T/2] \times \Omega)$.}
\end{equation}
Then, redatuming into $\Omega$ can be accomplished by constructing
the map $\curlyL$ using the data $\Lambda_{\Gamma}^{2T}$ and
$g|_{M(\Gamma,r)}$.  Thus the central focus of this paper is the
following problem:
\begin{namedProblem}{P}
  Given $\Lambda_\Gamma^{2T}$ and $g|_{M(\Gamma,r)}$, determine the map $\curlyL$.
\end{namedProblem}
In Section \ref{sec:MovingData} we develop an algorithm to solve problem
(P) constructively.

%% \textbf{ WHERE DOES THIS GO?: A fundamental feature of our approach is
%%   that, generally, these synthesized waves will contain interactions
%%   with the unknown parts of the wavespeed, and thus they could not be
%%   obtained from direct simulation using only the known part of the
%%   wavespeed.  }

Our primary motivation for studying the problem (P) stems from the
fact that it arises as a step in several variations of the Boundary Control (BC) method, see \cite{Belishev1987} for the original formulation of the method. In
theory, the BC method allows one to reconstruct $(M,g)$ given
$\Lambda_\Gamma^{2T}$ for $T > \max_{x \in M} d(x,\Gamma)$. This
reconstruction is based on a layer stripping argument, for which the
first step is to recover $g$ in the semigeodesic coordinates of
$\Gamma$. 
As these coordinates do not cover the whole $M$,
we refer to this procedure as the {\em local recovery step}. 
The second step is to
solve the redatuming problem (P), and consequently we refer to this
step as the {\em redatuming step}. Solving Problem (P)
allows one to propagate the data $\Lambda_\Gamma^{2T}$ into the interior
of $M$ and thus enables one to repeat the local recovery step with
data in the interior.
By alternating between the local recovery and redatuming steps, one can
reconstruct the Riemannian structure $(M,g)$ further and further away
from $\Gamma$. In particular, one can reconstruct the structure
outside the domain where the semigeodesic coordinates of $\Gamma$ are
applicable. 

Such an alternating iteration has been used in several uniqueness results for inverse boundary value problems \cite{Isozaki2010, Katchalov2001,Kurylev2015,Lassas2014},
however, the iteration is unstable, and it has not been
implemented computationally to our knowledge. In order to understand
how to regularize the iteration, 
we need to study the inherent instability of the local recovery and redatuming steps. 
The present paper considers the redatuming step,
that is, the problem (P), while we have previously studied the
local recovery step \cite{Hoop2016}.

We divide our redatuming procedure into two steps,
which we call {\em moving receivers} and {\em sources}, respectively.
The moving receivers step 
concerns solving the following time-windowed problem:
\begin{namedProblem}{WP}
  Given $\Lambda_\Gamma^{2T} f|_{(T-r,T+r) \times
    \Gamma}$ for $f \in L^2([0,T-r]\times\Gamma)$, determine
  $u^f(T,\cdot)$ in $\Omega$. Here, $g$ is known in
  $M(\Gamma,r)$.
\end{namedProblem}
% Figure:
\begin{figure}[!htb]
  \centering
  \includegraphics[width=1.4in]{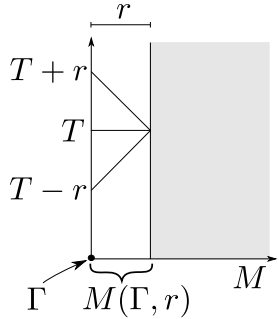}
  \caption{
    \label{img:WindowingCartoon}
    Geometry of the windowed problem (WP). The shaded region indicates
    where the metric is unknown.  
  }
\end{figure}
Time-windowing arises naturally in the redatuming problem, and it also allows us to consider the problem (WP) as a unique continuation problem for the wave equation on $(T-r,T+r) \times M(\Gamma, r)$.
We illustrate the geometry of (WP) in Figure
\ref{img:WindowingCartoon}.
Let us note that, as $f$ is assumed to be supported on $[0,T-r] \times \overline{\Gamma}$, $u^f$ satisfies the homogeneous Neumann boundary condition on $(T-r,T+r) \times \partial M$. In our computational procedure, we will allow $f$ to have support in $[T - r, T ] \times \Gamma$. This does not affect the stability properties of the moving receivers step, since if $f \in L^2([T - r, T ] \times \Gamma)$, then solving (\ref{eqn:1stForwardProb}) in $M$ to obtain $u^f(T,\cdot)$ is a classical well-posed problem, when $g$ is known on $M(\Gamma, r)$.

We will show that, after a transposition, the moving sources step reduces to a problem analogous to (WP). For this reason, we develop stability theory only for the moving receivers step. 

The problem (WP) is a special case of the following unique continuation problem
\begin{namedProblem}{UC}
  Given Cauchy data $(u, \p_\nu u)$ on $(T-r,T+r)
  \times \Gamma$, determine $u(T,\cdot)$ near $\Gamma$. Here, $u$ satisfies $\p_t^2 u - \Delta_g u = 0$, and $g$ is known in
  $M(\Gamma,r)$.
\end{namedProblem}
%% Let us point out that Tataru's theorem says that the problem (UC) has
%% a positive solution when the metric $g$ is smooth \cite{Tataru1995},
%% and that this has also been shown to hold in the piecewise smooth case
%% \cite[see e.g.]{Kirpichnikova2012}. 
Thus, the stability of
(WP) can be no less favorable than that of (UC). On the other
hand, since problem (WP) considers waves that satisfy a global Neumann
boundary condition, while no such boundary conditions are imposed in
(UC), it is not immediately evident how the stability of (WP) compares
to that of (UC). 
Nonetheless, we will show that (WP) enjoys the same stability as
(UC), and we present sharp stability theory for the problem (WP) in
Section \ref{sec:Stability}.
%\note{Need to say something about geometry of $\Gamma$ here}

Let us briefly summarize the stability theory. Under suitable
conditions, the problem (UC) is known to be conditionally H\"{o}lder
stable, see e.g.  \cite[Thm. 3.2.2]{Isakov2010}. We give a geometric
reformulation of this result in terms of convexity of
$\Gamma$, and show that conditional H\"{o}lder stability is optimal
for (UC). Our counterexample establishing the optimality of H\"{o}lder
type stability works in the case of strictly convex $\Gamma$, and moreover, we
show that a refined version of this counterexample also works in the
case of the windowed problem (WP). In particular, this shows that the
global homogeneous Neumann boundary condition on $(T-r,T+r) \times \p M$ in (WP)
does not improve the stability. This should be contrasted with
\cite{Bardos1996}, where unconditional Lipschitz stability is obtained
for a problem of the form (WP),
with strictly convex $\Gamma$, under the additional assumption
that $u^f(T,\cdot)$ and $\p_t u^f(T,\cdot)$ are supported near $\Gamma$.

Unique continuation problems have been studied from computational point of view, for example, the so-called quasireversibility method has been used to solve (UC) in \cite{Klibanov1992}. 
In this paper we propose to use the iterative time-reversal control method due to Bingham et. al. \cite{Bingham2008} to solve (WP). 
In \cite{Bingham2008} this method was applied to the coefficient determination problem to find $g$ given $\Lambda_\Gamma^{2T}$, however, as explained in Section \ref{sec:MovingData}, it can be used to solve (WP) as well.
We describe also the moving sources step in Section \ref{sec:MovingData} 
and give there a complete algorithm solving (P). 
Finally, we give computational examples in Section \ref{sec_comp}.
To our knowledge, this is the first computational implementation of the iterative time-reversal control method.

\section{Stability Theory for the Windowed Problem}

\label{sec:Stability}

In this section, we consider the stability theory for the windowed
problem (WP). We begin by recalling the stability theory for the more
general problem (UC). We were not able to find all the results in
Sections 2.1-3 in the literature, however, the techniques used there
are well-known.

\subsection{Conditional H\"older stability for UC under convexity conditions}

%% \note{THE NOTATION AND PROOFS MUST BE CHECKED CAREFULLY!}
%% \note{TODO: WHO DO WE CITE FOR \ref{eqn:CarlemanEstimate}, Isakov?}

We use ideas from \cite{Isakov2010, Stefanov2013} to prove the
following conditional H\"older stability estimate:

\begin{lemma}
\label{lemma:HolderTypeStability}
Let $T > 0$, $x_0 \in \Gamma$, and suppose that $\Gamma \subset \p M$ is strictly convex
in the sense of the second fundamental form. 
Then there exist a neighbourhood 
$U$ of $(0,x_0)$ in $(-T,T) \times M$,
$\kappa \in (0,1)$ and $C > 0$
such that for all $u \in H^2((-T,T) \times M)$
satisfying $\p_t^2 u - \Delta_g u = 0$
it holds that 
\begin{equation}
\label{eqn:HolderTypeStability}
\|u\|_{H^1\left(U\right)} \leq C(F + A^{1 - \kappa} F^{\kappa}),
\end{equation}
where 
$F =\|u\|_{H^{3/2}((-T,T) \times \Gamma)} + \|\p_\nu u\|_{H^{1/2}((-T,T) \times \Gamma)}$ and 
$A = \|u\|_{H^1((-T,T) \times M)}$.
\end{lemma}  
\begin{proof}
Let $\Gamma' \subset \Gamma$ be a coordinate neighbourhood of $x_0$,
let $s_0 > 0$ be small, and set $\Omega = (-s_0,s_0) \times \Gamma'$.
We will use semigeodesic coordinates $(s,y) \in \Omega$ associated to
$\Gamma'$.  Here a point $x \in M$ near $x_0$ has the coordinates
$(s,y)$ where $y$ is the closest point to $x$ in $\Gamma'$ and $s =
d(x, y)$.  Furthermore, we choose the coordinates so that $x_0 =
(0,0)$, and extend $g$ smoothly to $\Omega$. All norms, inner
products, gradients and Hessians will be taken with
respect to the Riemannian structure associated with $g$ on $\Omega$.

Let $Q = (-T_0,T_0) \times \Omega$ for some $T_0 \in (0,T]$. We recall
  that if a function $\phi$ is strongly pseudo-convex in
  $\overline{Q}$ with respect to the wave operator $P := \p_t^2
  -\Delta_g$, then, for $v \in C_0^\infty(Q)$ one has the Carleman
  estimate \cite[Thm. 28.2.3]{Hormander1985a}:
\begin{equation}
  \label{eqn:CarlemanEstimate}
  \tau \int_{Q} e^{2 \tau \phi} (|\p_t v|^2 + |\nabla v|^2 + \tau^2 |v|^2) dt dx
  \le C \int_{Q} e^{2 \tau \phi} |Pv|^2 dt dx, \quad \tau > 1.
\end{equation}
By approximation, this estimate also extends to $v \in H_0^2(Q)$ .

To obtain a function $\phi$ that is strongly pseudo-convex in
$\overline{Q}$ , we follow the approach from \cite{Stefanov2013}.
Specifically, we construct a function $\psi$ satisfying:
\begin{enumerate}[label=(\roman*), topsep=4pt]
  \item \label{itm:NotLightLike} $|\p_t \psi| \neq |\nabla \psi|$ in $Q$,
  \item \label{itm:HamFlow} $H_p^2 \psi > 0$ on
    $T^*(\overline{\Omega})\setminus 0$ whenever $\psi = p = H_p \psi
    = 0$,
\end{enumerate}
where, $H_p$ denotes the Hamiltonian flow associated with principal
symbol $p$ of $P$. If $\psi$ satisfies
\ref{itm:NotLightLike}-\ref{itm:HamFlow}, then the function $\phi :=
\exp(\beta \psi) - 1$ will be strongly pseudo-convex in
$\overline{Q}$, provided that $\beta \gg 1$. Moreover, when $\psi(t,x)
= \theta(t) + \rho(x)$, condition \ref{itm:HamFlow} is equivalent to
\begin{align}
\label{pseudo_convexity}
\p_t^2 \theta + D^2 \rho(\xi,\xi) > 0, \quad 
\xi \in S_x M,\ (t,x) \in Q,
\end{align}
holding whenever $\p_t \theta \pm (\xi, \nabla \rho) = 0$, see e.g.
\cite{Stefanov2013}. Here, we use $S_x M$ to denote the unit sphere at $x$.

In order to derive (\ref{eqn:HolderTypeStability}) from
(\ref{eqn:CarlemanEstimate}) via a cut-off argument, the function
$\psi$ needs to be chosen so that it decays when the distance to the
origin $(t,s,y) = (0,0,0)$ grows in the region $s > 0$.
%Moreover, we need for some $\epsilon > 0$ that 
%\begin{align}
%\label{level_set_cond}
%\{(t,y);\ \phi(t,0,y) \ge \phi(0,0,0) - \epsilon \} \subset (-T,T) \times \Gamma'.
%\end{align}
%need the level sets $\phi$
%$U_\epsilon \cap ((-T,T) \times \Gamma' $.
Let $R, \delta, \mu > 0$, and 
consider the polynomial
\begin{equation*}
  \psi(t,s,y) = (R-s)^2 - \delta t^2 - \mu |y|^2 - R^2.
\end{equation*}
Here, we identify $y$ with its coordinate representation and use $|y|$
to denote the Euclidean length of the coordinate vector for $y$. The
function $\psi$ decays as needed when $0 < s < R$. Let us show that
$R$, $\delta$ and $\mu$ can be chosen so that $\psi$ satisfies
(\ref{pseudo_convexity}).  Consider first the case $\mu = 0$. Then, on
the boundary, $s=0$ and this inequality reduces to
\begin{equation*}
  -\delta + \sigma^2 + R \SecondFundamentalForm(\eta, \eta) > 0 ,
  \quad \xi = (\sigma,\eta),
\end{equation*} 
where $\SecondFundamentalForm$ denotes the second fundamental form for
$\Gamma$.  By strict convexity, it holds that $R
\SecondFundamentalForm(\eta, \eta) \ge |\eta|^2$ for large enough $R >
0$.  Moreover, $\sigma^2 + |\eta|^2 = |\xi|^2 = 1$, and therefore
(\ref{pseudo_convexity}) holds if $\delta < 1$.  The inequality
(\ref{pseudo_convexity}) remains valid in a neighbourhood of the
origin for small $\mu > 0$ by smoothness. To show that
\ref{itm:NotLightLike} holds, we note that $\p_t \psi(0,0) = 0$ and
$|\nabla \psi(0,0)| \geq 2R$, thus $|\p_t \psi| \neq |\nabla \psi|$ at
the origin. Smoothness of $\psi$ implies that this condition also
holds in a neighborhood of the origin. Then, we can shrink $Q$ by
decreasing $s_0$, $T_0$, and $\Gamma'$, in order to ensure that $\psi$
satisfies \ref{itm:NotLightLike} and \ref{itm:HamFlow} on $Q$.

\begin{figure}[!htb]
  \centering
  \includegraphics[width=2.75in]{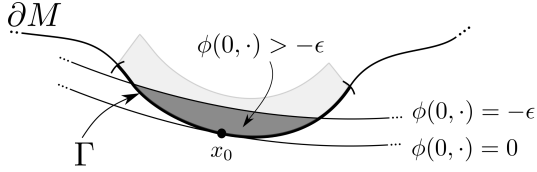}
  \caption{
    \label{img:HolderCartoon}
    A cartoon illustrating the geometry for Lemma
    \ref{lemma:HolderTypeStability} at $t = 0$. The lightly shaded
    region represents $Q^+ \cap \{t = 0\}$, while the dark region
    depicts $U(\epsilon) \cap \{t = 0\}$. For this simple case we have
    taken $\Gamma' = \Gamma$.}
\end{figure}

We write $Q^+ = Q \cap \{ s > 0\}$, and use the right inverse of the
trace map to get $w \in H^2(Q^+)$ with Cauchy data $(u, \p_\nu u)$ on
$(-T,T) \times \Gamma'$ satisfying $\|w\|_{H^2(Q^+)} \leq C F$. Then,
$v = u-w$ has zero Cauchy data on $(-T,T) \times \Gamma'$ and we
extend $v$ by zero as a function on $Q$. Then, $f = Pv = -Pw$ satisfies
$\norm{f}_{L^2(Q^+)} \le C F$. We note that $\psi < 0$ in $Q^+$ so $\phi
< 0$ there too. Choose $\epsilon > 0$ sufficiently small so that the
set
\begin{equation*}
  U(\epsilon) = \{(t,s,y) \in \R^{1+n};\ \phi(t,s,y) \ge -\epsilon,\ s > 0 \}
\end{equation*}
satisfies $U(\epsilon) \subset Q^+ $, and choose $\chi \in
C_0^\infty(Q)$ such that $\chi = 1$ in $U(\epsilon)$.  See Figure
\ref{img:HolderCartoon} for an illustration.

We will apply (\ref{eqn:CarlemanEstimate}) to $\chi v$. Note that 
\begin{equation*}
  P(\chi v) = \chi Pv + [P,\chi]v = \chi f + [P,\chi]v,
\end{equation*}
where the commutator $[P,\chi]$ is a first-order differential operator
that vanishes on the set $U(\epsilon)$.
Thus
\begin{align*}
  &\tau \int_{U(\epsilon/2)} e^{2 \tau \phi} \left(|\p_t v|^2 + |\nabla v|^2 + \tau^2 |v|^2\right) dt dx \\ 
  &\quad \le C \left(\int_{Q^+} e^{2 \tau \phi} |f|^2 dt dx 
  + \int_{Q^+ \setminus U(\epsilon)} e^{2 \tau\phi} |[P,\chi]v|^2 dt dx\right).
\end{align*}
Using that $\tau > 1$, and setting $p = 2\norm{\phi}_{L^\infty(Q)} +
\epsilon$, it holds that
\begin{equation*}
  \norm{v}_{H^1(U(\epsilon / 2))}^2 \le C \left(e^{\tau p}\norm{f}_{L^2(Q^+)}^2 + e^{-\tau \epsilon}\norm{v}_{H^1(Q^+)}^2\right).
\end{equation*}
Since $v = u + w$ and $\|w\|_{H^2(Q^+)} \leq CF$ we have that
$\|v\|_{H^1(Q^+)} \leq C(A + F)$. Recalling that $\|f\|_{L^2(Q^+)}
\leq CF$, we find:
\begin{equation*}
  \norm{v}_{H^1(U(\epsilon / 2))}^2 \le C \left(e^{\tau p}F^2 +
  e^{-\tau \epsilon}(A+F)^2\right).
\end{equation*}
Choosing $\tau$ as in
\cite[Thm. 3.2.2]{Isakov2010}, we obtain
\begin{equation*}
  \norm{v}_{H^1(U(\epsilon / 2))} \le C F^{\kappa} (A+F)^{1-\kappa},
\end{equation*}
where $\kappa = \epsilon / ( p + \epsilon)$. Then, since $0 < 1 -
\kappa < 1$, we see that $(A+F)^{1-\kappa} \leq A^{1-\kappa} +
F^{1-\kappa}$. Finally, we again use that $v = u + w$ and the bound on
$\|w\|_{H^2(Q^+)}$ to conclude:
\begin{equation*}
  \norm{u}_{H^1(U(\epsilon/2))} \le C(F +  A^{1-\kappa} F^{\kappa}).
\end{equation*}
\end{proof}

\subsection{Convexity is necessary for H\"older stability}

%% \begin{figure}[!htb]
%%   \label{img:HolderScalingCartoon}
%%   \centering
%%   \includegraphics[width=2.75in]{./draft_images/Cartoons/HolderRegionScaling.eps}
%%   \caption{Euclidean case, with $M$ a disk of radius $R > 0$. Taking
%%     $\Sigma$ to be a tangent line to $\p M$ yields that $\diam(\{\psi(0,\cdot) >
%%     \epsilon\}) = \mathcal{O}({\sqrt{\epsilon}}).$}
%% \end{figure}

\begin{figure}[!htb]
  \label{img:GaussianBeamCartoon}
  \centering
  \includegraphics[width=1.75in]{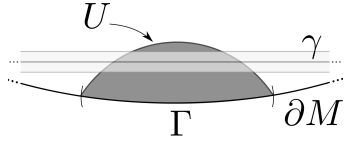}
  \caption{Example where a geodesic $\gamma$ passes through $U$ but
    fails to intersect $\Gamma$. The dark shaded region is a slice of
    $U$ at some time $t$, and the light shaded region is the spatial
    projection of the effective support of a Gaussian beam centered on
    $\gamma$.}
\end{figure}

In this section, we demonstrate that a convexity condition must hold
between the sets $\Gamma$ and $U$ in order for a H\"{o}lder stability
estimate of the type (\ref{eqn:HolderTypeStability}) to
hold. 
We follow ideas from \cite{Ralston1982}, and show  
that if there is a bicharacteristic ray
that passes over $U$ but does not meet $[-T,T]\times \overline{\Gamma}$,
then (\ref{eqn:HolderTypeStability}) can not hold. 

Let $\gamma$ be a unit
speed geodesic on $M$ and consider the corresponding bicharacteristic ray $\beta(t) = (t,\gamma(t))$. We suppose that there exists $t_0 \in
(-T,T)$ for which $(t_0,\gamma(t_0)) \in U$ but $\gamma(t) \not \in
\overline{\Gamma}$ for all $t \in [-T,T]$.
Let us consider a Gaussian beam $u_\epsilon$ concentrated on $\gamma$.
We refer to
\cite{Katchalov2001,Ralston1982} for the construction of Gaussian beams,
and recall here only that, for $\epsilon > 0$, $u_\epsilon$ is a family of solutions to the wave equation $\p_t^2 u_\epsilon -  \Delta_g u_\epsilon = 0$ on $(-T,T) \times M$ satisfying 
for any $j \in
\mathbb{N}_0$ and multi-index $\lambda \in \mathbb{N}_0^n$
\begin{equation*}
  |\p_t^j \p_x^\lambda (u_\epsilon(t,x) - \chi(t,x) U_\epsilon^N(t,x))| \leq C_{j,\lambda,N} \epsilon^{N - (j + |\lambda| + n/4)},
  \quad t \in (-T,T),\ x \in M.
\end{equation*}
Here, $\chi$ is a smooth function having small support around $\beta$ and satisfying $\chi = 1$ near $\beta$,
and in local coordinates $(t,z)$, $U_\epsilon^N$ is a smooth function of the form
$$
U_\epsilon^N(t,z) = \epsilon^{-n/4} \exp\left(i\epsilon^{-1} \Theta(t,z)\right) \sum_{j=0}^N \epsilon^j a_j(t,z), \quad N \in \N.
$$
Here, $\Theta$ is a complex valued function whose imaginary part vanishes on $\beta$, and satisfies $\Im \Theta(t,z) \ge \theta(t) |z-\gamma(t)|^2$ for some continuous strictly positive function $\theta$. Moreover, the function $a_0$ does not vanish on $\beta$.

First we discuss how the right hand side of
(\ref{eqn:HolderTypeStability}) behaves with respect to the family
$u_\epsilon$.
We define, for an integer $r > 0 $ and each $\epsilon
>0$, the quantities, 
$$
A_\epsilon := \|u_\epsilon\|_{H^r([-T,T]\times
  M)}, \quad F_\epsilon := \|u_\epsilon\|_{H^r([-T,T]\times\Gamma)} +
\|\p_\nu u_\epsilon\|_{H^{r-1}([-T,T]\times\Gamma)},
$$ 
and investigate
how they behave as $\epsilon \rightarrow 0$. 
Since $\beta$ does not intersect 
$[-T,T] \times \overline{\Gamma}$, we can
choose $\chi$ so that it vanishes on $[-T,T]\times \Gamma$. Then,
\begin{equation*}
  |\p_t^j \p_x^\lambda u_\epsilon(t,x)| \leq C_{j,\lambda,N} \epsilon^{N - (j + |\lambda| + n/4)}, \quad \text{for $(t,x) \in [-T,T]\times\Gamma$}.
\end{equation*}
Consequently, for any $r \in \mathbb{N}$, it holds that $F_\epsilon \leq C_{r,N} \epsilon^{N - (r + n/4)}$.
Now, let us consider the quantity $A_\epsilon$. 
We have $\norm{U_\epsilon^N}_{H^r([-T,T] \times M)} \le C_{r,N} \epsilon^{-r-n/4}$, and therefore also
$A_\epsilon \leq C_{r,N} \epsilon^{-r-n/4}$.
Thus, for any fixed $0 < \kappa < 1$, there exists a constant $C_{r,N} > 0$
such that
\begin{equation*}
  F_\epsilon + A_\epsilon^{1 - \kappa} F_\epsilon^{\kappa} \leq C_{r,N}
  \left(\epsilon^{N - (r + n/4)} + \epsilon^{\kappa N - (r + n/4)}\right).
\end{equation*}
Finally, we choose $N$ sufficiently large such that $\kappa N > r + n/4$,
and conclude that
\begin{equation}
\label{cex_rhs_vanishes}
  F_\epsilon + A_\epsilon^{1-\kappa} F_\epsilon^{\kappa} \rightarrow 0 \quad\text{as $\epsilon \rightarrow 0$.}
\end{equation}

We now consider how the left-hand side of
(\ref{eqn:HolderTypeStability}) behaves with respect to the family
$u_\epsilon$. 
In view of (\ref{cex_rhs_vanishes}),
it remains to show that $\|u_\epsilon\|_{H^r(U)}$ stays positive as $\epsilon \rightarrow 0$. Since $\|u_\epsilon\|_{H^r(U)} \geq \|u_\epsilon \|_{L^2(U)}$, we can consider only the $L^2$ norm. Let $B$ be a ball
containing the point $\gamma(t_0)$
and satisfying $\{t_0\} \times B \subset U$. On \cite[p. 176]{Katchalov2001}, it
is shown that $\lim_{\epsilon \rightarrow 0}
\|u_\epsilon(t,\cdot)\|_{L^2(B)} = a(t)$, where $a$ is
a continuous strictly positive function. 
Thus for small $\delta > 0$ it holds that $\lim_{\epsilon \rightarrow 0}
\|u_\epsilon \|_{L^2([-\delta, \delta] \times B)} > 0$.
This concludes the proof, showing that if there is a bicharacteristic ray
passing over $U$ that does not meet $[-T,T]\times \overline{\Gamma}$,
then (\ref{eqn:HolderTypeStability}) cannot hold.

\subsection{A counterexample to Lipschitz stability for UC}
\label{subsec:counterExForUC}

In this section, we give a counterexample showing that
(\ref{eqn:HolderTypeStability}) cannot hold with $\kappa = 1$.  This
example is a variation of the classical counterexample by Hadamard
\cite[p. 33]{Hadamard1953}, adapted to a strictly convex setting.

Let us consider a case where $M$ is contained in the half disk
$$
\{ r e^{i \theta} \in \C;\ r \in (0,1],\ |\theta|< \pi/2\}.
$$
We assume that $M$ is equipped with the Euclidean metric and suppose
that $\Gamma \subset \p M$ is of the form $\Gamma = \{ e^{i \theta}
\in \C;\ |\theta|< \theta_0\}$, for some $\theta_0 \in (0, \pi/2)$.

We consider a family of stationary waves in $M$. For $n \in \N$, we define
\begin{equation*}
   \phi_n(r e^{i \theta}) := r^{-n} e^{i n \theta}.
\end{equation*}
Then, we recall that in polar coordinates $(r,\theta) \mapsto r e^{i \theta}$
the Laplacian has the form
\begin{equation*}
  \Delta = \p_r^2 + r^{-1} \p_r + r^{-2} \p_\theta^2.
\end{equation*}
Using this formula, it is straightforward to check that the $\phi_n$
are harmonic in $M$ (note that $0 \not\in M$). Letting $u_n(t,x) =
\phi_n(x)$, it is immediate that $(\p_t^2 - \Delta) u_n = 0$ on $\R \times M$.

Next, we observe that $\phi_n(1,\theta) = e^{in\theta}$ and $\p_\nu \phi_n(1,
\theta) = i n e^{in\theta}$. Thus,
\begin{equation*}
  \|(\phi_n,\p_\nu \phi_n)\|_{H^k(\Gamma)\times H^{k-1}(\Gamma)} \leq  \|\phi_n\|_{H^k(\Gamma)} +  n \|\phi_n\|_{H^{k-1}(\Gamma)} \sim n^k.
\end{equation*}
%% OLD:
%% \begin{equation*}
%%   \norm{(\phi_n,\p_\nu \phi_n)}_{H^1 \times L^2(\Gamma)} \sim n.
%% \end{equation*}
Then, we let $\epsilon, \theta_1, s > 0$ be small and define the sets $\Omega =
(1-\epsilon,1) \times (-\theta_1,\theta_1)$ and $U = (T-s,T+s) \times
\Omega$. We note that, if $\theta_1 \le \theta_0$ and $\epsilon$ is
sufficiently small, then $\Omega \subset M$. Letting $q =
1-\epsilon$, we observe that
$$
\norm{\phi_n}_{L^2(\Omega)}^2 = 
\int_{q}^1 r^{-2n}  rdr \int_{-\theta_1}^{\theta_1} |e^{i n \theta}|^2 d\theta \sim \frac{q^{-2(n-1)}}{n-1}
$$ for large $n > 0$.  Thus, a Lipschitz stability estimate of
the form
%% \note{Choose $s$ explicitly here. For small $\epsilon$, $\theta_1 <
%%   \theta_0 << \pi/2$ we can choose $s$ so that $U \subset M(\Gamma,s)$
%%   and $M(\Gamma,s)$ does not intersect the origin. Also it would be
%%   more elegant to argue directly here (rather than via
%%   contradiction). Compare to the analogous half-space case below.}
\begin{align*}
  &\norm{u}_{H^k(U)} \le C \norm{(u,\p_\nu u)}_{H^k \times
    H^{k-1}((T-s,T+s) \times \Gamma)}
\end{align*}
leads to a contradiction when we take $u = u_n$. To see this, we first note that the left-hand side is bounded below by
%the $L^2$ norm of $u|_{t=T}$, hence it is bounded below by 
$C q^{-(n-1)}/\sqrt{n-1}$, where $C$ is independent of $n$. This holds
because $\|u_n\|_{L^2(U)} = 2s\|\phi_n\|_{L^2(\Omega)}$. On the other
hand, the right-hand side of this inequality is comparable to $n^{k}$.
Thus, we get the contradiction that
\begin{equation*}
  q^{-(n-1)} \lesssim n^k\sqrt{n-1}
\end{equation*}
for large $n$.

\subsection{A counterexample to Lipschitz stability for WP}

In this section, we construct a counterexample to Lipschitz type
stability for the problem (WP) in the strictly convex boundary
setting. Our construction is based on finding a family of Neumann
sources $\{f_n\}$ producing a family of waves $\{u^{f_n}\}$ solving
(\ref{eqn:1stForwardProb}) that exhibit similar stability properties
to the waves considered in Section \ref{subsec:counterExForUC}. The
waves $u^{f_n}$ will then satisfy the hypotheses of (WP), and show
that Lipschitz type stability does not hold for (WP).  We carry out
our construction in two steps.  First, for some $\epsilon > 0$, we
find waves $u_n$ with vanishing Neumann traces that behave like the
waves from Section \ref{subsec:counterExForUC} on $t \in
[T-\epsilon,T+\epsilon]$.  Then, we use exact controllability to
obtain Neumann sources $f_n \in L^2([0,T-\epsilon]\times\Gamma)$ that
reproduce these waves, in the sense that $u^{f_n}(t,\cdot) =
u_n(t,\cdot)$ for $t \geq T-\epsilon$.

We consider the case where $M$ is the unit disk equipped with the
Euclidean metric, and $\Gamma = (-\theta_0, \theta_0)$ for some
$\theta_0 \in (\pi/2,\pi)$ in polar coordinates. Let $0 < \epsilon <
T$ and $\Omega \subset M$ a neighborhood of $M(\Gamma,\epsilon)$, and
select $\epsilon$ and $\Omega$ sufficiently small that $(0,0) \not \in
\overline{\Omega}$. We will make use of a fixed cut-off function $\chi \in
C^\infty([0,2T]\times M)$ which we choose to have the form $\chi(t,x)
= \chi_t(t) \chi_x(x)$ with $\chi_t \in C^\infty([0,2T])$ and $\chi_x
\in C^\infty(M)$. In particular, we choose $\chi_t$ so that it satisfies
$\chi_t = 1$ on a neighborhood of $[T-\epsilon,2T]$ and $\chi_t = 0$
on a neighborhood of $[0,T-2\epsilon]$. Also, we choose $\chi_x$
to satisfy $\chi_x \equiv 1$ on $M(\Gamma,\epsilon)$ and $\chi_x
\equiv 0$ on $M \setminus \Omega$.

Let $\phi$ be any harmonic function in $\Omega$. Using $\phi$, we
define $w$ to be the solution to:
\begin{align*}
  (\partial_t^2  - \Delta) w(t,x) &= 0, & (t,x) \in (0, 2T) \times M,\\ 
  \boundaryConds w(t,x) &= \boundaryConds(\chi(t,x) \phi(x)), & (t,x) \in(0,2T) \times \p M, \\
  w(T,x) = \chi_x \phi,\quad\partial_t w(T,x) &= 0,  & x \in M,
\end{align*}
and, let $v$ solve
\begin{align*}
  (\partial_t^2 - \Delta) v(t,x) &= 0, & (t,x) \in (0, 2T) \times M,\\ 
  \boundaryConds v(t,x) &= \boundaryConds(\chi(t,x) \phi(x)), & (t,x) \in(0, 2T) \times \p M, \\
    v(T-2\epsilon,x) = \partial_t v(T-2\epsilon,x) &= 0, & x \in M.
\end{align*}
We define $u = w - v$ and study the properties of $u$ in terms of $w,
v$ and $\phi$. Let us observe that $\p_\nu u =
0 $ on $(0,2T)\times \p M$, since $\p_\nu w$ and $\p_\nu v$ coincide
there.

To begin our analysis of $u$, we show that $w(t,x) = \phi(x)$ for
$(t,x) \in K$, where
\begin{equation*}
  K = \left\{(t,x) \in [T-\epsilon,T+\epsilon] \times M(\Gamma,\epsilon)~:~d(x, M \setminus M(\Gamma,\epsilon)) > \left|t - T\right|\right\}.
\end{equation*}
In order to show that $w = \phi$ on $K$, let us abuse notation and
identify $\phi$ with its constant extension in time. Then, we note
that $\phi$ is harmonic in $\Omega$ and constant in time, thus,
$(\p_t^2 - \Delta)\phi(t,x) = (\p_t^2 - \Delta)w(t,x) = 0$ on
$(T-\epsilon,T+\epsilon) \times \Omega$.  Next, we note that $w(T,
\cdot) = \phi$ and $\p_tw(T,\cdot) = \p_t \phi = 0$ in
$M(\Gamma,\epsilon)$. Finally, we observe that $\p_\nu w = \p_\nu
\phi$ on $[T-\epsilon, T+\epsilon] \times (\p M \cap
M(\Gamma,\epsilon))$, since $\chi = 1$ there. Thus, finite speed of
propagation for (\ref{eqn:1stForwardProb}) implies that $w$ and $\phi$
coincide in $K$.

We define $\Sigma = [T-\epsilon,T+\epsilon] \times \Gamma$ and $\Sigma' =
[T-2\epsilon,T+2\epsilon] \times \p M$. Then, for a set $U \subset
[T-\epsilon,T+\epsilon] \times M$, we investigate how the size of $u$
on $U$ compares to the size of $(u,\p_\nu u)$ on $\Sigma$. To that
end, we note that $\p_\nu u = 0$ on $\Sigma$ and observe that
\begin{equation*}
  \|(u,\p_\nu u)\|_{H^k(\Sigma)\times H^{k-1}(\Sigma)} =
  \|u\|_{H^k(\Sigma)} \leq \|w\|_{H^k(\Sigma)} + \|v\|_{H^k(\Sigma)}.
\end{equation*}
We will bound the norms on the right in terms of norms of
$\phi$. First, we bound the $H^k$ norm of $v$. Since $\chi_t$ is
identically zero on a neighborhood of $T - 2\epsilon$, $\p_\nu v =
\p_\nu(\chi \phi)$ vanishes identically on a neighborhood of
$\{T-2\epsilon\} \times \p M$. Because $v(T-2\epsilon, \cdot) = \p_t
v(T-2\epsilon,\cdot) = 0$, we see that $v$ satisfies compatibility
conditions to all orders at $t = T-2\epsilon$. Appealing to standard
estimates for the wave equation and trace theorems, we can then show
that
\begin{equation*}
  \|v\|_{H^k(\Sigma')} \lesssim \|\p_\nu v\|_{H^\ell(\Sigma')},
\end{equation*}
where $\ell > k$ (in particular $\ell = 2k+1$ works).  Combining this
with the previous estimate and using that $\p_\nu v = \p_\nu w$ on
$\Sigma'$ yields:
\begin{equation*}
  \|(u,\p_\nu u)\|_{H^k(\Sigma) \times H^{k-1}(\Sigma)} \lesssim
  \|w\|_{H^k(\Sigma)} + \|\p_\nu w\|_{H^\ell(\Sigma')}.
\end{equation*}
Then, since $\p_\nu w = \p_\nu (\chi \phi)$, $w = \phi$ on $K$, and
both $\Sigma \subset K$ and $\Sigma \subset \Sigma'$ we conclude:
%%%%%%%%%%%%%%%%%%%%%%%%%%
%% OLD:
%%%%%%%%%%%%%%%%%%%%%%%%%%
%% \begin{equation*}
%%   \|(u,\p_\nu u)\|_{H^1(\Sigma) \times L^2(\Sigma)} \lesssim \|\Phi\|_{H^1(\Sigma')} + \|\p_\nu \Phi\|_{L^2(\Sigma')}  \lesssim \|(\Phi,\p_\nu \Phi)\|_{H^1(\Sigma') \times L^2 (\Sigma')}.
%% \end{equation*}
\begin{equation*}
  \|(u,\p_\nu u)\|_{H^k(\Sigma) \times H^{k-1}(\Sigma)} \lesssim
  \|\phi\|_{H^k(\Sigma')} + \|\p_\nu \phi\|_{H^\ell(\Sigma')}.
\end{equation*}
Next, we let $q = 1 - \frac{\epsilon}{2}$ and note that $0 < q <
1$. We consider the set
\begin{equation*}
  U := \left\{(t,r,\theta)~:~T-\frac{\epsilon}{2} < t < T+\frac{\epsilon}{2},~~q < r < 1,~~\theta \in \Gamma\right\}.
\end{equation*}
Observe that $U \subset K$, so $w = \phi$ on $U$. Again, using
standard estimates for the wave equation and that $\p_\nu v =
\p_\nu(\chi \phi)$ on $\Sigma'$, we can show that
\begin{equation*}
  \|v\|_{H^k(U)} \lesssim \|\p_\nu v\|_{H^{2k}(\Sigma')} \lesssim
  \|\p_\nu \phi\|_{H^{2k}(\Sigma')}.  
\end{equation*}
Whereas, $w = \phi$ on $U$, so $\|w\|_{H^k(U)} = \|\phi\|_{H^k(U)}$.

Let us now take $\phi = \phi_n$ from the preceding section, and note that
$\phi_n$ is harmonic in $\Omega$. We let $w_n,v_n$ and $u_n$ denote the
waves associated with $\phi_n$ as constructed above. Then, the
estimates given in Section \ref{subsec:counterExForUC} imply that, for
any $j \in \mathbb{N}$, $\|\phi_n\|_{H^j(\Sigma')} \sim n^j$ and
$\|\p_\nu \phi_n\|_{H^j(\Sigma')} \sim n^{j+1}$, while
$\|\phi_n\|_{H^j(U)} \geq \|\phi_n\|_{L^2(U)} \gtrsim
q^{-(n-1)}/\sqrt{n-1}$. So,
\begin{equation*}
  \|u_n\|_{H^k(U)} = \|w_n - v_n\|_{H^k(U)} \gtrsim \left|\|\phi_n\|_{H^k(U)} - \|\p_\nu\phi_n\|_{H^{2k}(\Sigma')} \right|  \gtrsim q^{-(n-1)}/\sqrt{n-1}.
\end{equation*}
On the other hand,
\begin{equation*}
  \|(u_n,\p_\nu u_n) \|_{H^k(\Sigma) \times H^{k-1}(\Sigma)} \lesssim
  \|\phi_n\|_{H^k(\Sigma')} + \|\p_\nu \phi_n\|_{H^\ell(\Sigma')} \lesssim n^k + n^{\ell+1}.
\end{equation*}
Thus, a Lipschitz stability estimate of the form $\|u\|_{H^k(U)} \leq C
\|(u,\p_\nu u)\|_{H^k(\Sigma) \times H^{k-1}(\Sigma)}$ leads to the
contradiction that for all $n$,
\begin{equation*}
  q^{-(n -1)} \lesssim (n^k + n^{\ell+1})\sqrt{n-1}.
\end{equation*}
%Which, once again, is false since $0 < q < 1$.
  
We now show that, if $\tau := T-\epsilon$ is sufficiently large, there
exists a Neumann source $f_n \in L^2([0,T-\epsilon] \times \Gamma)$
for which $u^{f_n} = u_n$ on $[T-\epsilon,T+\epsilon] \times M$. To
see this, we first recall that $M$ is the unit disk equipped with the
Euclidean metric and that $\Gamma$ contains a neighborhood of the
half-circle $\theta \in (-\pi/2,\pi/2)$. This setting is considered on
p. 1030 of \cite{Bardos1992}, where it is noted that if $\tau > 6$,
then any bicharacteristic ray beginning above a point $x \in M$ will
pass over $[0,\tau]\times\Gamma$ in a non-diffractive point. Thus by
choosing $T$ large enough that $\tau = T - \epsilon > 6$, the
hypotheses of \cite[Th. 4.9]{Bardos1992} for exactly controlling $M$
from $[0,\tau] \times \Gamma$ will be satisfied. Specifically, the map
$f \mapsto (u^f(\tau,\cdot), \p_t u^f(\tau,\cdot))$ taking
$L^2([0,\tau] \times \Gamma) \rightarrow H^1(M) \times L^2(M)$ is
surjective (see \cite[ex. 2]{Bardos1992}, p. 1059). It is
straightforward to check that $(u_n(T-\epsilon,\cdot), \p_t
u_n(T-\epsilon, \cdot)) \in H^1(M) \times L^2(M)$, thus there exists a
source $f_n \in L^2([0,T-\epsilon] \times \Gamma)$ for which
$(u^{f_n}(T-\epsilon,\cdot), \p_t u^{f_n}(T-\epsilon,\cdot)) =
(u_n(T-\epsilon,\cdot), \p_t u_n(T-\epsilon, \cdot))$. Finally, we
note that the Cauchy data of $u^{f_n}$ and $u_n$ agree at $t =
T-\epsilon$, and the Neumann traces of both $u^{f_n}$ and $u_n$ vanish
on $[T-\epsilon, T+\epsilon] \times \p M$. Hence, uniqueness for
solutions to (\ref{eqn:1stForwardProb}) implies that
$u^{f_n}|_{[T-\epsilon, T+\epsilon] \times M} = u_n|_{[T-\epsilon,
    T+\epsilon] \times M}$.

To conclude, we have constructed a family of waves $\{u^{f_n}\}$ that
satisfy the hypotheses of (WP). Because these waves coincide with the
waves in the family $\{u_n\}$ on both $U$ and $\Sigma$, we see that a
Lipschitz type stability estimate cannot hold for (WP).
 
%% Let us also observe that the waves $u$ described in this section are
%% in fact smooth in $C$, hence conditional Lipschitz stability cannot be
%% obtained by imposing a smoothness hypothesis. To see that $u = w - v$
%% is smooth in $C$, we demonstrate that $w$ and $v$ are smooth there. To
%% see that $w$ is smooth over $C$, note that $w = \phi$ on $C$ and
%% recall that $\phi$ is a stationary wave which is harmonic (hence
%% smooth) over $\Omega$. To show that $v$ is also smooth on $C$, first
%% recall that the boundary data and Cauchy data for $v$ are
%% smooth. Likewise, the boundary and Cauchy data for $v$ vanish
%% identically at $t = T-\epsilon$. Thus compatibility conditions hold to
%% all orders for $v$, and we see that $v$ is smooth in $[T-\epsilon,T]
%% \times M$,  hence on $C$.

\section{Redatuming}
\label{sec:MovingData}

In this section, we present our redatuming procedure, which gives a
constructive solution to (P). We begin in subsection \ref{subsec:BCM}
by briefly reviewing
concepts from the iterative time reversal control method \cite{Bingham2008}.
As discussed in the introduction,
our approach to redatuming is
accomplished in two steps: subsection
\ref{subsec:MovingRec} is devoted to moving receivers, while
subsection \ref{subsec:MovingSrc} is devoted to moving sources.

\subsection{Notation and techniques}
\label{subsec:BCM}
The Riemannian Volume measure on $M$ is denoted by $dV$, and $dS$ will
denote the associated surface measure on $\p M$. When we evaluate $L^2$
inner products, the corresponding integrals will be evaluated with
respect to these measures. 

We define the \emph{control map}, which
is defined for $f \in L^2([0,T] \times \Gamma)$, by,
\begin{equation}
    W^Tf := u^f(T,\cdot).
\end{equation}
We recall that $W^T$ is bounded,
\begin{equation}
  W^T : L^2([0,T] \times \Gamma) \rightarrow L^2(M),
\end{equation}
which follows from \cite{Lasiecka1991}.
Now we form
the \emph{connecting operator},
\begin{equation}
  K^T := (W^T)^*W^T.
\end{equation}
The operator $K^T$ derives its name from the fact that it connects
inner-products between waves in the interior to measurements made on
the boundary. That is, for $f,h \in L^2([0,T] \times \Gamma)$,
\begin{equation}
  \langle u^f(T,\cdot),u^h(T,\cdot) \rangle_{L^2(M)} = \langle W^T f, W^Th \rangle_{L^2(M)} = \langle K^T f,h \rangle_{L^2([0,T] \times \Gamma)}.
\end{equation}
An essential fact about $K^T$ is that it can be obtained by processing
the boundary data, $\Lambda_{\Gamma}^{2T}$. Specifically, one can
construct $K^T$ via the Blagove{\v{s}}{\v{c}}enski{\u\i} identity, which
we use in a form analogous to the expression found in
\cite{Oksanen2013},
\begin{equation}
  \label{eqn:StandardBlagoIdentity}
  K^T =  J^T \Lambda^{2T}_{\Gamma} \Theta^T - R^T \Lambda^{T}_{\Gamma} R^T J^T \Theta^T.
\end{equation}
Here, the operators $R^T$, $J^T$, and $\Theta^T$ are defined as
follows: the time reversal operator,
$R^{T} : L^2([0,T] \times \Gamma) \rightarrow L^2([0,T] \times
\Gamma)$, is defined by
\begin{equation}
  R^T f (t, \cdot) = f(T - t, \cdot), \qquad \text{ for $0 < t < T$ },
\end{equation}
the time filtering operator,
$J^T : L^2([0,2T] \times \Gamma) \rightarrow L^2([0,T] \times
\Gamma)$, is given by,
\begin{equation}
  J^T f (t, \cdot) = \frac{1}{2} \int_t^{2T - t} f(s,\cdot) \,ds, \qquad \text{ for $0 < t < T$},
\end{equation}
and the zero extension operator,
$\Theta^T : L^2([0,T] \times \Gamma) \rightarrow L^2([0,2T] \times
\Gamma)$, is given by,
\begin{equation}
  \Theta^T f(t, \cdot) = \left\{
    \begin{array}{ll}
      f(t, \cdot) & 0 \leq t \leq T \\
      0 & T < t \leq 2T. \\ 
    \end{array}
  \right.
\end{equation}

In addition, we will use the restriction, $\rho^T : L^2([0,2T] \times
\Gamma) \rightarrow L^2([0,T] \times \Gamma)$, given by $\rho^T f =
f|_{[0,T] \times \Gamma}$. We will also use, for $r \in [0,T]$, the
family of orthogonal projections $P_r^T : L^2([0,T] \times \Gamma)
\rightarrow L^2([T-r,T] \times \Gamma)$, which too are obtained by
restriction. Lastly, for $s > 0$, we will use time delay operators,
given by
\begin{equation}
  Z_s \phi(t, \cdot) := \left\{
  \begin{array}{cl} 
    0 & \text{for $t \in [0,s]$} \\
    \phi(t-s, \cdot) & \text{for $t > s$}.
  \end{array}
  \right.
\end{equation}

We will need analogous operators defined on spaces of the form
$L^2([0,S] \times A)$, where $A \subset \overline{M}$ and $S > 0$. For
those operators, we use similar notation. For instance, we will also
write $R^S$ to denote the time-reversal operator on $L^2([0,S] \times
A)$. We note that, in all cases, our notation will only indicate the
appropriate final time $S$, since all four operators $R,J,\Theta,\rho$
act essentially in the temporal domain. We do not indicate the spatial
domain in our notation since it will be evident from context.

Finally, in some longer equations, we will suppress the spatial
dependence of functions in our notation. For example, let $F :
[0,T]\times M \rightarrow \R$ and $t \in \R$. Then, we will
occasionally write $F(t)$ to denote $F(t,\cdot)$.

\subsection{Moving Receivers}
\label{subsec:MovingRec}

In this section, we will construct the map $L$,
%, which we recall is defined by
\begin{equation}
  \label{eqn:definingLAgain}
  L : f \mapsto u^f|_{[0,T] \times M(\Gamma,r)}, \quad \text{for $f
    \in L^2([0,T] \times \Gamma)$.}
\end{equation}
We refer to the procedure for constructing $L$ as
\emph{moving receivers}, since evaluating $L$ is tantamount to
extrapolating receivers into $M(\Gamma,r)$.  Moving receivers is
accomplished through Algorithm \ref{algo:movingReceivers}, and we
demonstrate the correctness of this algorithm via Lemma
\ref{thm:MovingRec}. We note that Lemma \ref{thm:MovingRec} is
essentially demonstrated in \cite[Lemma 7]{Bingham2008}. However, we
repeat the proof here, since it is constructive and forms the basis
for Algorithm \ref{algo:movingReceivers}.

\begin{algorithm}
  \begin{algorithmic}
    \FOR{$f \in C_0^\infty([0,T] \times \Gamma)$}
    \FORALL{$0 < t < T$}
    \FORALL{$0 < \alpha$}
    \STATE Let : $h = h_{\alpha,t}$ denote the solution to 
    \begin{equation*}
      P_r^T(K^T + \alpha) P_r^T h = P_r^T K^T Z_{T - t}f
    \end{equation*}
    \STATE Solve : the wave equation in $[T-r, T] \times M(\Gamma,r)$
    to obtain $u^{h_{\alpha,t}}(T,\cdot)$
    \ENDFOR
    \STATE Compute : 
    \begin{equation*}
      Lf(t) = u^f(t,\cdot)|_{M(\Gamma,r)} = \lim_{\alpha \rightarrow 0} u^{h_{\alpha,t}}(T,\cdot).
    \end{equation*}
    \ENDFOR
    \ENDFOR
  \end{algorithmic}
  \caption{
    \label{algo:movingReceivers}
    Continuum level moving receivers procedure.}
\end{algorithm}

\begin{lemma}
  \label{thm:MovingRec}
  The map $L$ can be constructed from the data $\Lambda_{\Gamma}^{2T}$
  and the known sub-manifold $(M(\Gamma,r),g)$. Furthermore, $L$ is a
  bounded operator,
  \begin{equation}
    L : L^2([0,T] \times \Gamma) \rightarrow L^2(M(\Gamma,r) \times [0,T]).
  \end{equation}
\end{lemma}
\begin{proof}%[of Theorem \ref{thm:MovingRec}]
  We first note that the continuity of $L$ is demonstrated in
  \cite[Thm 2.0.0]{Lasiecka1991}, where it is shown that the map $f
  \mapsto u^f$ is bounded from $L^2([0,T] \times \Gamma) \rightarrow
  H^\beta([0,T] \times M)$ for $\beta = 3/5 - \epsilon$ and any
  $\epsilon > 0$. Since $H^\beta([0,T] \times M) \subset L^2([0,T]
  \times M)$ for $0 < \epsilon < 3/5$, and $M(\Gamma,r) \subset M$, it
  follows that $L$ is bounded.

  Because $L$ is bounded and $C_0^\infty([0,T]\times \Gamma)$ is dense in
  $L^2([0,T]\times \Gamma)$, it will suffice to show that $Lf$ can be
  constructed for any smooth $f$. We let $f \in C_0^\infty([0,T] \times \Gamma)$,
  and obtain $Lf$ by computing wavefield snapshots $Lf(t) = u^f(t,
  \cdot)|_{M(\Gamma,r)}$ for $t \in [0,T]$. To get $Lf(t)$, we first
  construct a family of sources $h_{\alpha,t} \in
  L^2([T-r,T]\times\Gamma)$ satisfying
  \begin{equation}
    \label{eqn:limiting_bndry_waves}
    \lim_{\alpha \rightarrow 0} u^{h_{\alpha,t}}(T,\cdot)|_{M(\Gamma,r)}
    = u^f(t,\cdot)|_{M(\Gamma,r)},
  \end{equation}
  where the limit is taken in $L^2(M(\Gamma,r))$. Since
  $\supp(h_{\alpha,t}) \subset [T-r,T]\times\Gamma$, finite speed of
  propagation for (\ref{eqn:1stForwardProb}) implies that $\supp(u^{h_{\alpha,t}}(s,\cdot)) \subset
  M(\Gamma,r)$ for $0 \leq s \leq T$. Thus, the waves
  $u^{h_{\alpha,t}}(T,\cdot)$ can be evaluated by solving
  (\ref{eqn:1stForwardProb}) in $M(\Gamma,r)$, and the wavefield
  snapshot $Lf(t)$ can be obtained from the limit
  (\ref{eqn:limiting_bndry_waves}).

  We now recall how the sources $h_{\alpha,t}$ can be obtained using
  the data $\Lambda_{\Gamma}^{2T}$. As in \cite{Bingham2008}, we
  consider the Tikhonov minimization problem,
  \begin{equation}
    \label{eqn:regMinBKLS}
    h_{\alpha,t} := \argmin_{h\in L^2([T-r,T]\times\Gamma)} \|u^h(T,\cdot) - u^f(t,\cdot) \|^2 + \alpha \|h\|^2.
  \end{equation}
  Since the operator $\p_t^2 - \Delta_g$ commutes with time
  translations, $u^f(t,\cdot) = u^{Z_{T-t}f}(T,\cdot)$. Using
  the operators defined above, we can rephrase (\ref{eqn:regMinBKLS})
  in the form
  \begin{equation}
    h_{\alpha,t} = \argmin_{h\in L^2([T-r,T]\times\Gamma)} \|W^TP h - W^TZ_{T-t}f \|^2 + \alpha \|h\|^2,
  \end{equation}
  where we have written $P = P_r^T$ to avoid some notational clutter.  Since
  the operator $W^T$ is bounded, \cite[Thm. 2.11]{Kirsch2011} implies
  that the unique solution to (\ref{eqn:regMinBKLS}) is given by:
  \begin{equation}
    h_{\alpha,t} = \left(P^*(W^T)^*W^T P + \alpha\right)^{-1} (W^T P)^* W^TZ_{T-t}f. 
  \end{equation}
  Because $K^T = (W^T)^*W^T$, we can rewrite this as,
  \begin{equation}
    \label{eqn:fAlphaBCM}
    h_{\alpha,t}  = (P K^T P + \alpha)^{-1} P K^TZ_{T-t}f.
  \end{equation}
  Since the operator $K^T$ can be constructed via the
  Blagove{\v{s}}{\v{c}}enski{\u\i} identity
  (\ref{eqn:StandardBlagoIdentity}), expression (\ref{eqn:fAlphaBCM})
  shows that $h_{\alpha,t}$ can be obtained from the data
  $\Lambda_{\Gamma}^{2T}$. 

  Finally, we show that (\ref{eqn:limiting_bndry_waves}) holds. We
  recall, if $g$ is smooth, that Tataru's Theorem \cite{Tataru1995}
  implies that $W^T P$ has dense range in $L^2(M(\Gamma,r))$, and that this
  also holds if $g$ is piece-wise smooth
  \cite{Kirpichnikova2012}. Thus, \cite[Lemma 1]{Oksanen2013} implies
  that $W^TP h_{\alpha,t} \rightarrow W^TZ_{T-t}f$ as $\alpha
  \rightarrow 0$. Hence, the sources $h_{\alpha.t}$ satisfy
  (\ref{eqn:limiting_bndry_waves}), which is what we wanted to show.
  \qquad
\end{proof}

\subsection{Moving Sources}
\label{subsec:MovingSrc}
As stated above, we refer to the procedure for constructing $\curlyL$ from
$L$ as \emph{moving sources}. We present the moving sources procedure
as Algorithm \ref{algo:movingSources} and demonstrate its validity in
Lemma \ref{thm:MovingSrc}.

\begin{algorithm}
  \begin{algorithmic}
    \FOR{$F \in C_0^\infty([0,T/2] \times \Gamma)$}
    \FORALL{$0 < t < T/2$}
    \FORALL{$0 < \alpha$}
    \STATE Let : $h = h_{\alpha,t}$ denote the solution to 
    \begin{equation*}
      P_r^{T/2}(K^{T/2} + \alpha) P_r^{T/2} h = P_r^{T/2} \curlyK^* Z_{T/2 - t}F,
    \end{equation*}
    where $\curlyK$ is given by (\ref{eqn:BlagoTypeInteriorOperatorExpr}).
    \STATE Solve : the wave equation in $[T/2-r,T/2] \times M(\Gamma,r)$ to obtain $u^{h_{\alpha,t}}(T/2,\cdot)$
    \ENDFOR
    \STATE Compute :
    \begin{equation*}
      \curlyL F(t) = w^F(t,\cdot)|_{M(\Gamma,r)} = \lim_{\alpha \rightarrow 0} u^{h_{\alpha,t}}(T/2,\cdot)
    \end{equation*}
    \ENDFOR
    \ENDFOR
  \end{algorithmic}
  \caption{
    \label{algo:movingSources}
    Continuum level moving sources procedure.}
\end{algorithm}

We show that $\curlyL$ can be constructed from $L$ via a transpostion
argument. With that goal in mind, let us introduce a final value problem
that coincides with the time-reversal of (\ref{eqn:2ndForwardProb}),
\begin{align}
  \label{eqn:FinalValueProblem}
  (\partial_t^2 - \LaplaceBeltrami) v(t,x) &= H(t,x), & (t,x) \in (0, T) \times M,\\ 
  \boundaryConds v(t,x) &= 0, & (t,x) \in(0,T) \times \p M \notag\\
  v(T,\cdot) = \partial_t v(T,\cdot) &= 0, & x \in M. \notag
\end{align}
Here, $H \in L^2([0,T] \times M(\Gamma,r))$, and we denote the
solution to (\ref{eqn:FinalValueProblem}) by $v^H$. We have the
following result concerning the transpose of $L$.

\begin{lemma}
  \label{lemma:Ltranspose}
  Let $F \in L^2([0,T] \times M(\Gamma,r))$, then,
  \begin{equation}
    \label{eqn:LAdjoint}
    R^T L^* R^T F =  w^F
    |_{[0,T]\times\Gamma}.
  \end{equation}
  %%%%%%%%%%%%%% Not sure why this is here?
  % Where we recall that $\p_\nu$ denotes the outward co-normal
  % derivative on $\p M$.
\end{lemma}
\begin{proof}
  We first note that by \cite[Thm 2.0.0]{Lasiecka1991}, the map $F
  \mapsto v^F|_{[0,T] \times \Gamma}$ is bounded from $L^2([0,T]
  \times M(\Gamma,r)) \rightarrow H^\beta([0,T] \times \Gamma)$ where
  $\beta = 3/5$. Thus it is also a bounded operator mapping $L^2([0,T]
  \times M(\Gamma,r)) \rightarrow L^2([0,T] \times \Gamma)$. Since the
  map $F \mapsto w^F|_{[0,T] \times \Gamma}$ is the time reversal of
  this map, it is also bounded.
    
  To prove (\ref{eqn:LAdjoint}), we let $F \in C_0^{\infty}([0,T] \times
  M(\Gamma,r))$, $h \in C_0^\infty([0,T] \times \Gamma)$, and argue by
  density. Using the divergence theorem, the fact that $u^h$ solves
  (\ref{eqn:1stForwardProb}), and that $v^F$ solves
  (\ref{eqn:FinalValueProblem}), we see,
  \begin{align*}
    \langle F, &Lh \rangle_{L^2([0,T] \times M(\Gamma,r))} = \langle F, u^h \rangle_{L^2([0,T] \times M)} \\ 
    &= \langle (\p_t^2 - \Delta_g) v^F, u^h \rangle_{L^2([0,T] \times M)} - \langle v^F, (\p_t^2 - \Delta_g) u^h\rangle_{L^2([0,T] \times M)}\\
    &= \langle -\pND v^F, u^h \rangle_{L^2([0,T] \times \p M)} - \langle v^F, -\pND u^h\rangle_{L^2([0,T] \times \p M)}\\
    &= \langle v^F, h \rangle_{L^2([0,T] \times \Gamma)}.
  \end{align*}
  On the last line, we have used (\ref{eqn:LAdjoint}) and the support
  properties of $F$ and $h$. By the density of $C_0^\infty$ spaces in
  their respective $L^2$ spaces and the boundedness of the operator
  $L$, we conclude that
  \begin{equation}
    \label{eqn:LTranspose}
    L^* F = v^F|_{[0,T] \times \Gamma}.
  \end{equation}

  Let us denote $R = R^T$ and show that $Rv^{RF} = w^F$. To see this,
  we first note that:
  \begin{equation*}
      (\p_t^2 - \LaplaceBeltrami)(Rv^{RF})(t) = (\p_t^2 -
    \LaplaceBeltrami)v^{RF}(T - t) = RF(T - t) = F(t).
  \end{equation*}
  Furthermore, $\p_t (Rv^{RF})(0) = - \p_t v^{RF}(T - 0) = -\p_t
  v^{RF}(0) = 0,$ and $(Rv^{RF})(0) = v^{RF}(T) = 0$. Finally,
  $(Rv^{RF})|_{[0,T] \times \p M} = R((v^{RF})|_{[0,T] \times \p M}) =
  0$. Hence, $Rv^{RF}$ solves (\ref{eqn:2ndForwardProb}) with
  right-hand side $F$. By uniqueness of solutions to
  (\ref{eqn:2ndForwardProb}), it follows that $w^F = Rv^{RF}$. Thus,
  in conjunction with (\ref{eqn:LTranspose}),
  \begin{equation*}
    RL^*RF = Rv^{RF}|_{[0,T] \times \Gamma} = w^F|_{[0,T] \times
      \Gamma},
  \end{equation*}
  which is what we wanted to show. \qquad
\end{proof} 

Next, we introduce a Blagove{\v{s}}{\v{c}}enski{\u\i} type identity
relating the inner-product between $w^F(T/2,\cdot)$ and
$u^h(T/2,\cdot)$ to an inner-product between $F$ and an operator
applied to $h$. We remark that our proof follows an analogous strategy
to the technique used to derive (\ref{eqn:StandardBlagoIdentity}).
\begin{lemma}
  Let $F \in L^2([0,T/2] \times M(\Gamma,r))$ and $h \in L^2([0,T/2]
  \times \Gamma)$. Then,
  \begin{equation}
    \label{eqn:BlagoTypeInteriorBoundary}
    \begin{array}{l}
      \langle w^F(T/2,\cdot),u^h(T/2,\cdot) \rangle_{L^2(M)} = \langle F,\curlyK h\rangle_{L^2([0,T/2] \times
        M(\Gamma,r))},\\
    \end{array}
  \end{equation}
  where,
  $\curlyK : L^2([0,T/2] \times \Gamma) \rightarrow L^2([0,T/2]\times
  M(\Gamma,r))$ is bounded and can be constructed by,
  \begin{equation}
    \label{eqn:BlagoTypeInteriorOperatorExpr}
    \curlyK  = J^{T/2} L \Theta^{T/2} - \rho^{T/2} R^TLR^T \Theta^{T/2}J^{T/2}\Theta^{T/2}.
  \end{equation}
\end{lemma}
%% \textbf{TODO: It is possible to show that $\rho^{T/2} R^TLR^T
%%   \Theta^{T/2} = R^{T/2} L^{T/2} R^{T/2}$ , where $L^{T/2}$ is the
%%   obvious restriction of $L$. Do we want to do this?}
\begin{proof}
  To simplify our notation, for this proof we let $R = R^T$, $J =
  J^{T/2}$, $\Theta = \Theta^{T/2}$, and $\rho = \rho^{T/2}$.

  To see that $\curlyK$ is bounded, let us write $\Wint^{T/2} : F
  \mapsto w^F|_{[0,T/2] \times M(\Gamma,r)}$. Then $\Wint^{T/2}$ is
  bounded by \cite{Lions1972a}. By definition, $\curlyK =
  (\Wint^{T/2})^* W^{T/2}$, hence $\curlyK$ is bounded since it is a
  composition of bounded operators.
  %% TODO: DOUBLE CHECK THAT THIS IS THE CORRECT CITATION

  Since $\curlyK$ is bounded, we  argue by density. Let
  $F \in C_0^\infty([0,T/2] \times M(\Gamma,r))$ and $h \in
  C_0^\infty([0,T/2] \times \Gamma)$.  Because we are interested in
  obtaining the inner-product $\langle
  w^F(T/2,\cdot),u^h(T/2,\cdot)\rangle_{L^2(M)}$, we will consider the
  family of inner-products $I(t,s) := \langle
  w^F(t,\cdot),u^h(s,\cdot)\rangle_{L^2(M)}$, parametrized with $0\leq
  s \leq T$ and $0 \leq t \leq T/2$. We note that this quantity
  behaves like a one-dimensional wave with a forcing term:
  \begin{align*}
    (\p_t^2 - \p_s^2)I(t,s) &=(\p_t^2 - \p_s^2) \langle w^F(t), u^h(s)\rangle_{L^2(M)} \\
                            &=\langle \Delta_g w^F(t) + F(t), u^h(s) \rangle_{L^2(M)} - \langle w^F(t), \Delta_g u^h(s)\rangle_{L^2(M)},
  \end{align*}
  since $w^F$ and $u^h$ solve (\ref{eqn:2ndForwardProb}) and
  (\ref{eqn:1stForwardProb}) respectively. Next, we apply the
  divergence theorem, use Lemma \ref{lemma:Ltranspose} and the fact
  that $\p_\nu w^F = 0$, and appeal to the support properties of $F$
  and $h$, to find
  \begin{align*}
    (\p_t^2 - \p_s^2)I(t,s) &= \langle F(t), u^h(s) \rangle_{L^2(M)} - \langle w^F(t), \p_\nu u^h(s) \rangle_{L^2(\p M)}\\
    &=\langle F(t), L  \Theta h(s) \rangle_{L^2(M(\Gamma,r))} - \langle RL^* R \Theta F(t), \Theta h(s)\rangle_{L^2(\Gamma)}.
  \end{align*}
  Then, we note that $I(0,\cdot) = \p_t I(0,\cdot) = 0$, since
  $w^F(0,\cdot) = \partial_t w^F(0,\cdot) = 0$.  Thus $I$ solves an
  inhomogeneous one dimensional wave equation in the rectangle $(t,s)
  \in (0,T/2)\times(0,T)$, with unit wavespeed and vanishing initial
  conditions. By finite speed of propagation, the boundary condition
  at $s = 0$ does not affect the solution $I(t,s)$ when $s \geq t$.
  Hence, for $s \geq t$ we can solve for $I(t,s)$ by Duhamel's
  principle,
  \begin{equation}
    I(t,s) = \frac{1}{2} \int_0^t \int_{s-(t-\tau)}^{s+(t-\tau)} \langle F(\tau), L \Theta h(\sigma)\rangle_{L^2(M(\Gamma,r))} - \langle RL^* R\Theta F(\tau), \Theta h(\sigma) \rangle_{L^2(\Gamma)} \,d\sigma\,d\tau .
  \end{equation}
  Setting $s = t = T/2$ we see,
  \begin{align*}
    I(T/2&,T/2) = \langle w^F(T/2), u^h(T/2) \rangle_{L^2(M)} \\
               &= \frac 1 2 \int_0^{T/2} \int_{t}^{T - t} \langle F(t), L \Theta h(s)\rangle_{L^2(M(\Gamma,r))} - \langle RL^* R \Theta F(t), \Theta h(s)\rangle_{L^2(\Gamma)} \,ds \,dt\\
               &= \left\langle F, J L \Theta h\right)_{L^2([0,T/2] \times M(\Gamma,r))} - \langle RL^* R  \Theta F, \Theta J \Theta h \rangle_{L^2([0,T/2] \times \Gamma)}\\
               &= \left\langle F, (J L \Theta - \rho RLR\Theta J \Theta) h\right\rangle_{L^2([0,T/2] \times M(\Gamma,r))} .
  \end{align*}
  Thus we conclude that
  $\curlyK = JL\Theta - \rho R L R \Theta J \Theta$.  \qquad
\end{proof}

\begin{lemma}
  \label{thm:MovingSrc}
  The map $\curlyL$ can be constructed from the operator $L$ and the
  known sub-manifold $(M(\Gamma,r),g)$. Moreover, $\curlyL$ is a
  bounded operator,
  \begin{equation}
    \curlyL : L^2([0,T/2]\times M(\Gamma,r)) \rightarrow
    L^2([0,T/2]\times M(\Gamma,r)).
  \end{equation}
\end{lemma}
\begin{proof}%[of Theorem \ref{thm:MovingSrc}]

  We begin by noting that the boundedness of $\curlyL$ is known, see
  e.g.  \cite{Lions1972a}.
  
  %%%%%%%%%%%%%%%%%%%%%%%%%%%%%%%%%%%%%%%%%%%%%%%%%%%%%%%%%%%%%%
  %% Note:
  %% see reference p. 267, Gronwall's
  %%%%%%%%%%%%%%%%%%%%%%%%%%%%%%%%%%%%%%%%%%%%%%%%%%%%%%%%%%%%%% 
  
  We will ultimately need to obtain $\curlyK^*$, and any method to
  transpose $\curlyK$ will suffice. However, we remark that evaluating
  $\curlyK^*$ by transposing the operator expression
  (\ref{eqn:BlagoTypeInteriorBoundary}) would require one to construct
  $L^*$, which would entail a similar cost to constructing $\curlyK^*$
  itself. We give an efficient method to evaluate $\curlyK^*$ in
  Section \ref{sec:MovingSrcsModification}.

%  To compute the transpose of $\curlyK$, we let $\{\varphi_i\}$ be an
%  orthonormal basis for $L^2([0,T/2] \times \Gamma)$ and $\{\psi_j\}$
%  an orthonormal basis for $L^2([0,T/2] \times M(\Gamma,r))$. Then the
%  matrix for $\curlyK^*$ with respect to these bases is simply
%  $\curlyK^*_{j,i} = (\curlyK^* \psi_j, \varphi_i) = (\psi_j,\curlyK
%  \varphi_i)$, and the last inner product can be computed from
%  \ref{eqn:BlagoTypeInteriorBoundary}. Thus $\curlyK^*$ can be
%  computed.

The strategy that we use to construct $\curlyL$ follows a similar
pattern to the method which we used to construct $L$. For
a source $F \in C_0^\infty([0,T/2]\times M(\Gamma,r))$ and time $t \in
[0,T/2]$, we will obtain the wavefield snapshot $\curlyL F(t) =
w^F(t,\cdot)|_{M(\Gamma,r)}$ by finding a family of sources
$h_{\alpha,t} \in L^2([0,T/2-r]\times \Gamma)$ for which
$u^{h_{\alpha,t}}(T/2,\cdot) \rightarrow
w^F(T/2,\cdot)|_{M(\Gamma,r)}$. We then evaluate
$u^{h_{\alpha,t}}(T/2,\cdot)$ by solving (\ref{eqn:1stForwardProb}) in
$[0,T/2]\times M(\Gamma,r)$ and obtain $\curlyL F(t)$ by taking the
limit as $\alpha \rightarrow 0$.

Let $\alpha > 0$. To obtain the source $h_{\alpha,t}$ we consider the
following Tikhonov problem,
\begin{equation}
  h_{\alpha,t} := \argmin_{h \in L^2([T/2-r,T/2] \times \Gamma)} \|u^h(T/2,\cdot) - w^F(t,\cdot) \|_{L^2(M)}^2 + \alpha \|h\|^2.
\end{equation}
We note that this regularized control problem is structurally similar
to the problem (\ref{eqn:regMinBKLS}), however the present problem has
a control time of $T/2$ and its target state is $w^F(t,\cdot)$. Thus
by the argument given in the proof of Lemma \ref{thm:MovingRec}, this
problem has a unique solution, $h_{\alpha,t}$, given by,
\begin{equation}
  h_{\alpha,t} = ((W^{T/2}P)^*(W^{T/2}P) + \alpha I)^{-1} (W^{T/2}
  P)^* w^F(t,\cdot),
\end{equation}
where we have written $P$ in place of $P_r^{T/2}$ for notational
clarity. Now, we note that $w^F(t,\cdot) = w^{Z_{T/2 -
    t}F}(T/2,\cdot)$, so we can use equation
(\ref{eqn:BlagoTypeInteriorBoundary}) to conclude that
$(W^{T/2})^*w^F(t,\cdot) = \curlyK^* Z_{T/2 - t}F$. Since
$(W^{T/2}P)^*W^{T/2}P = PK^{T/2}P$, we find
\begin{equation}
  \label{eqn:IntSrcControl}
  h_{\alpha,t} = (P K^{T/2} P + \alpha I)^{-1} P\curlyK^*Z_{T/2 - t}F.
\end{equation}
Thus $h_{\alpha,t}$ can be obtained from known quantities. Finally, by
Lemma 1 in \cite{Oksanen2013}, we have that
$u^{h_{\alpha,t}}(T/2,\cdot)|_{M(\Gamma,r)} \rightarrow
w^F(T/2,\cdot)|_{M(\Gamma,r)}$. \quad
\end{proof}

\section{Computational examples}
\label{sec_comp}

In this section, we present computational examples that demonstrate
both the receiver moving procedure discussed in Section
\ref{subsec:MovingRec} and the source moving procedure discussed in
Section \ref{subsec:MovingSrc}. We demonstrate our methods in a
conformally Euclidean setting, however, we stress that our techniques
can be applied in the general Riemannian setting.

\subsection{Forward modeling and discretization}

In our computational experiment, we take $M = \R \times [-1,0]$ with a
conformally Euclidean metric $g = c^{-2}dx^2$. For the wave-speed $c$,
we use $c(x,y) = 1-y$. We simulate waves propagating for $2T$ time
units, where $T = 2.0$, and make source and receiver measurements on
the set $[0,2T] \times \Gamma$, where $\Gamma = [-\ell,\ell] \times
\{0\} \subset \p M$ and $\ell = 3.1$.  The wave-speed $c$ is known in
Euclidean coordinates on the subset $M(\Gamma,r)$ where $r = 0.5$. Let
us point out that $\Gamma$ is strictly convex in the sense of the
second fundamental form of $(M,g)$.

For sources, we use a basis of Gaussian pulses of the form
\begin{equation*}
  \varphi_{i,j}(t,x) = C \exp\left(-a_t (t-t_{s,i})^2 -a_x
  (x-x_{s,j})^2\right),
\end{equation*}
%%%%%%%%%%%%%%%%%%%%%%%%%%%%%%%%%%%%%%%%%%%%%%%%%%%%%%%%%%%%%%%%%
%%  [PAK] WE ROUNDED FROM:
%%  a_x = a_t = 1.381551055796427e+03
%%%%%%%%%%%%%%%%%%%%%%%%%%%%%%%%%%%%%%%%%%%%%%%%%%%%%%%%%%%%%%%%%
with parameters $a_t = a_x = 1.382 \cdot 10^3$, and we choose $C$ to
normalize the $\varphi_{i,j}$ in $L^2([0,T]\times \Gamma, \dtdS)$.
Sources are applied at regularly spaced points $(x_{s,j}, 0)$ with
$x_{s,j} = -3.0 + (j-1)\Delta x_s$ for $j = 1,\ldots,N_{x,s}$ and
times $t_{s,i} = 0.025 + (i-1) \Delta t_s$ for $i=1,\ldots,N_{t,s}$.
The source offset $\Delta x_s$ and time between source applications
$\Delta t_s$ are both taken to be $\Delta x_s = \Delta t_s = .025$. At
each of the $N_{x,s} = 241$ source positions we apply $N_{t,s} = 79$
sources. For each basis function, we record the Dirichlet trace data
at regularly spaced points $(x_{r,k},0)$ with $x_{r,k} = -3.1 +
(k-1)\Delta x_r$ for $k = 1,\ldots,N_{x,r}$ at times $t_{r,l} = (l-1)
\Delta t_r$ for $l=1,\ldots,N_{t,r}$.  The receiver offset $\Delta
x_r,$ satisfies $\Delta x_r = 0.5 \Delta x_s$ resulting in $N_{x,r} =
497$ receiver positions.  The time between receiver measurements,
$\Delta t_r$, satisfies $\Delta t_r = 0.1 \Delta t_s$, resulting in
$N_{t,r} = 1601$ measurements at each receiver position.

We discretize the Neumann-to-Dirichlet map by solving the forward
problem for each source $\varphi_{i,j}$ and recording its Dirichlet
trace at the receiver positions and times described above. That is, we
simulate the following data,
\begin{equation}
  \label{eqn:DiscreteNtD}
  \left\{ \Lambda_\Gamma^{2T}\varphi_{i,j}(t_{r,l},x_{r,k}) =
  u^{\varphi_{i,j}}(t_{r,l},x_{r,k}) :
  \begin{array}{l}
    i=1,\ldots,N_{t,s}, ~j = 1,\ldots,N_{x,s},\\ l=1,\ldots,N_{t,r},
    ~k = 1,\ldots,N_{x,r}
  \end{array}
  \right\}.
\end{equation}
To perform the forward modelling, we use a continuous Galerkin finite
element method with piecewise linear Lagrange polynomial elements and
implicit Newmark time-stepping. This is implemented using the FEniCS
package \cite{LoggMardalEtAl2012a}.

For $0 \leq \tau_1 < \tau_2 \leq T$ we define $S_{\tau_1}^{\tau_2} :=
\linearSpan\{\varphi_{i,j} : \tau_1 < t_{s,j} < \tau_2\}$, and let
$S^{\tau} = S_0^{\tau}$.  We note that, since the sources
$\varphi_{ij}$ are well localized in time, the space
$S_{\tau_1}^{\tau_2}$ serves as a finite dimensional substitute for
the spaces $L^2([\tau_1,\tau_2] \times \Gamma)$. Then, to apply the
moving receivers and moving sources procedures we need the operators
$K^\tau$ for $\tau = T$ and $\tau = T/2$ respectively. Thus, for $\tau
= T, T/2$ we discretize the connecting operator $K^{\tau}$ by
computing its action as an operator on $S^\tau$.  We accomplish this
by restricting the discrete Neumann-to-Dirichlet data,
(\ref{eqn:DiscreteNtD}), to $S^\tau$ and computing a discrete analog
of (\ref{eqn:StandardBlagoIdentity}). Specifically, we first compute
the Gram matrix $[G^\tau]_{ij} = \langle \varphi_i,\varphi_j
\rangle_{L^2([0,\tau]\times\Gamma, \dtdS)}$ and its inverse
$[G^{\tau}]^{-1}$. Then, for $A = J^\tau \Lambda_{\Gamma}^{2\tau}$,
$R^\tau\Lambda_{\Gamma}^{\tau}$ and $R^\tau J^\tau$, we compute the
matrix for $A$ acting on $S^\tau$ by:
\begin{equation*}
  [A]_{ij} = \sum_k [G^{\tau}]^{-1}_{ik} \langle\varphi_k, A
  \varphi_j\rangle_{L^2([0,\tau]\times\Gamma, \dtdS)}.
\end{equation*}
Finally, we use these matrices to compute the matrix for $K^\tau$:
\begin{equation}
  \label{eqn:discreteBlago}
  [K^\tau] = [J^\tau \Lambda_{\Gamma}^{2\tau}] - [R^\tau \Lambda_{\Gamma}^{\tau}] [R^\tau J^\tau].
\end{equation}

The control problems introduced in the moving receivers and moving
sources problems are posed over $L^2([\tau -r, \tau] \times \Gamma)$
for $\tau = T,T/2$ respectively. In both cases we must solve linear
problems of the form $(PK^\tau P + \alpha) h_\alpha = Pb$ where $b$ is
a function in $L^2([0,\tau] \times \Gamma)$ and $P$ is the projection
$P : L^2([0,\tau] \times \Gamma) \rightarrow L^2([\tau-r,\tau] \times
\Gamma)$. To approximate the action of $P$, we construct a mask $[P]$
that selects the indices belonging to $S_{\tau-r}^\tau$. We then
recast the control problem in the finite dimensional case by finding
the coefficient vector $[h_\alpha]$ for a function $h_\alpha \in
S_{\tau-r}^\tau$ satisfying:
\begin{equation}
  \label{eqn:discreteControl}
  ([P][K^\tau][P] + \alpha) [h_\alpha] = [P][b],
\end{equation}
where $[b]$ denotes the coefficients of the projection of $b$ onto
$S^{\tau}$. We solve (\ref{eqn:discreteControl}) using restarted GMRES
with an appropriate choice of $\alpha$, documented below.

The last step in both the moving receivers and moving sources
procedures is to solve (\ref{eqn:1stForwardProb}) with the source
$h_\alpha$ given by (\ref{eqn:discreteControl}) in order to compute
$u^{h_\alpha}(\tau,\cdot)|_{M(\Gamma,r)}$. To do this, note that
$h_\alpha \in S_{\tau-r}^\tau$, so $h_\alpha$ is effectively supported
in $[\tau-r,\tau] \times \Gamma$. Thus by finite speed of propagation
and the fact that $c$ is known in $M(\Gamma,r)$ we can compute
$u^{h_\alpha}(\tau,\cdot)$ by solving (\ref{eqn:1stForwardProb}) using
the same computational scheme as used to generate
(\ref{eqn:DiscreteNtD}) and then restricting the result to
$M(\Gamma,r)$.

\subsection{Computational implementation of moving receivers}

We now specialize the preceding discussion to the moving receivers
setting. For this problem, we want to compute an approximation to
$u^f(t,\cdot)|_{M(\Gamma,r)}$ for $t \in [0,T]$ and $f \in L^2([0,T]
\times \Gamma)$. By Lemma \ref{thm:MovingRec}, the control problem
we must solve for this procedure is a discrete version of
(\ref{eqn:fAlphaBCM}). Thus the parameters for the discrete control
problem (\ref{eqn:discreteControl}) are $\tau = T$ and $b = K^T
Z_{T-t} h$. So we let $h_{\alpha,t}$ denote the solution to:
\begin{equation}
  \label{eqn:discreteMvngRec}
  ([P] [K^T] [P] + \alpha) [h_{\alpha,t}] = [P] [K^T] [Z_{T - t} f].
\end{equation} 
We finally approximate  $u^f(t,\cdot)|_{M(\Gamma,r)}$ by computing
$u^{h_{\alpha,t}}(T,\cdot)$, as described after
(\ref{eqn:discreteControl}).

%%\note{TODO: CHECK THE VALUE OF ALPHA AND SPECIFY SPACING} 
For the discrete moving sources procedure we need a discrete version
of $L$. We partially discretize $L$ by applying the moving receivers
procedure to each the basis functions $\varphi_{1,j} \in S^T$, at
regularly spaced times $t_l = 0, \Delta t_s, ..., T$ and saving the
receiver measurements on a regularly spaced grid of points $p_k \in
[-\ell,\ell] \times [0, r] \subset M(\Gamma,r)$, where the spacing
between adjacent $p_k$ is equal to $\Delta x_s$ in both
directions. More explicitly, we let $h_{jl}$ denote the solution to
(\ref{eqn:discreteMvngRec}) with $f = \varphi_{1,j}$, $t = t_l$, and
$\alpha = 10^{-4}$. We then compute the wave $u^{h_{jl}}(T,\cdot)$
approximating $u^{\varphi_{1,j}}(t_l,\cdot)$ in $M(\Gamma,r)$ and save
the values of $u^{h_{jl}}$ at the points $p_k$. In total, we compute
the following data:
\begin{equation}
  \label{eqn:DiscreteL}
  \left\{ L \varphi_{1,j}(t_{l},p_{k}) = u^{h_{jl}}(T,p_k) : 
  \begin{array}{ll}
    j = 1,\ldots,N_{x,s}, &l=1,\ldots,N_{t,s}, \\
    k = 1,\ldots,N_{p} & \\
  \end{array}
  \right\}.
\end{equation}
Note that we do not explicitly compute $L\varphi_{i,j}$ for $i >
1$. We avoid carrying out these computations because
$L\varphi_{i,j}(t_l,\cdot) = L\varphi_{1,j}(t_{l-(i-1)},\cdot)$ for $l
\geq i$ and $L\varphi_{i,j}(t_l,\cdot) = 0$ for $l < i$. This follows
because the time between source applications coincides with the
temporal spacing between measurement times and because the wave
equation is time translation invariant.  Thus it would be redundant to
compute $L\varphi_{i,j}$ for all $i > 1$. Moreover, storing every such
value would increase the amount of data by a factor of $N_{t,s}$,
which would be prohibitively costly. 
%In fact, in our implementation we
%could not compute a complete matrix representation of $L$ and hold it
%in RAM. 
%Instead, we work with the data (\ref{eqn:DiscreteL}) and
%modify our moving sources procedure accordingly. We outline our
%implementation in subsection
%\ref{sec:MovingSrcsModification}. 

Finally, we mention that for the discrete version of the moving
sources procedure, we must compute inner-products between
$L\varphi_{ij}$ and certain functions in $L^2([0,T] \times
M(\Gamma,r))$. To approximate these integrals we use a tensor product
of trapezoidal rules on the data (\ref{eqn:DiscreteL}).
%% \note{TODO: Also mention something about volume form here}

\subsection{Moving receivers example}

We provide an example to demonstrate our moving receivers
procedure. For a source we use:
\begin{equation*}
  f(t,x) = \exp\left(-((t - t_c)^2 +(x - x_c)^2)/\sigma^2\right),
\end{equation*} 
with parameters $t_c = 0.25$, $x_c = 0.0$, and $\sigma = 0.1$.  We
solve (\ref{eqn:discreteMvngRec}) with $\alpha = 5 \cdot 10^{-5}$ for
several times $t$, and compare the results with the true wavefields in
Figure \ref{tbl:ResultImages}. Since $r = 0.5$, we note that, for $t >
0.5$, it would not be possible to directly simulate $u^f(t,\cdot)$
without knowing the metric in the complement of $M(\Gamma,r)$. Thus
the wavefield snapshots depicted in Figure \ref{tbl:ResultImages} with
$t > 0.5$ could not be directly simulated under our assumption that
the wave-speed is only known in $M(\Gamma,r)$. Of particular interest
are the snapshots with $t \geq 1.25$. There, we observe a reflection
off $\p M \setminus \Gamma$ that has traveled through the unknown set
$M \setminus M(\Gamma,r)$ before returning to the known set
$M(\Gamma,r)$, yet our moving receivers procedure was able to capture
this reflected wave-front.

\begin{table}[htb!]
\centering
%\begin{tabular}{c{2cm} c{2.2in} c{2.2in}}
\begin{tabular}{c c}
\toprule
True wavefield $u^f(t,\cdot)$ 
& Approximate wavefield $u^{h_{\alpha,t}}(T,\cdot)$ \\
\midrule
\includegraphics[width=2.3in]{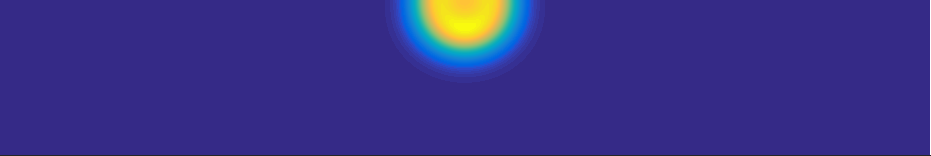}
& 
\includegraphics[width=2.3in]{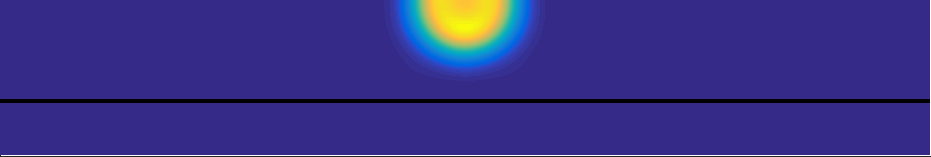}
\\
\includegraphics[width=2.3in]{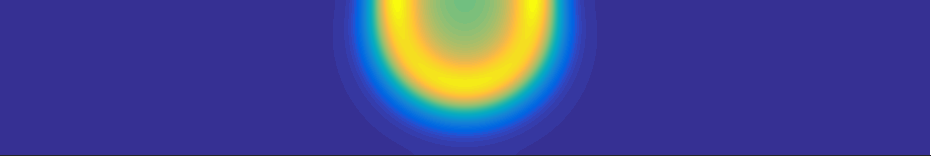}
& 
\includegraphics[width=2.3in]{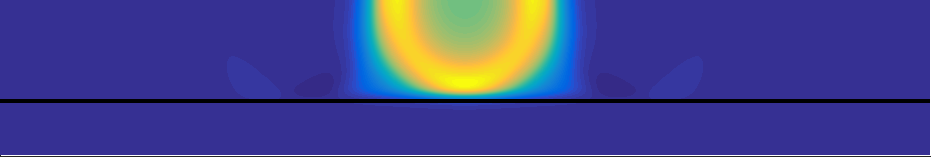}
\\
\includegraphics[width=2.3in]{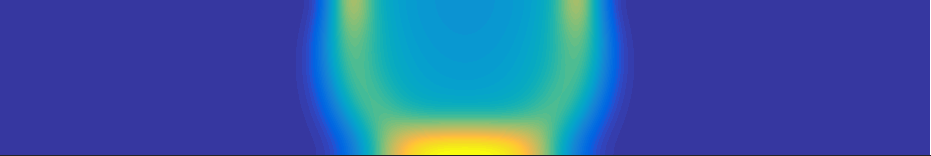}
& 
\includegraphics[width=2.3in]{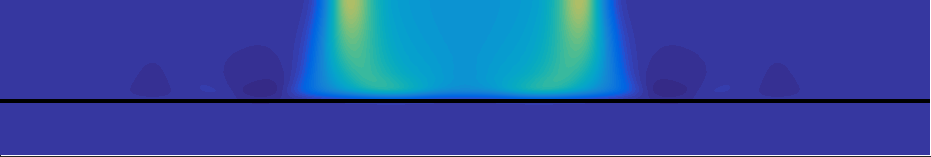}
\\
\includegraphics[width=2.3in]{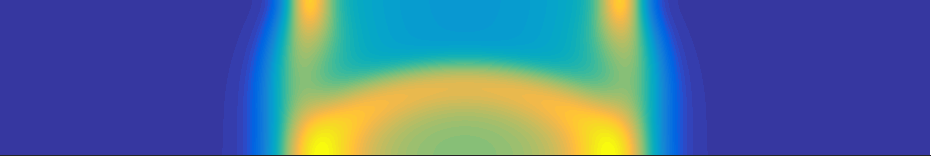}
& 
\includegraphics[width=2.3in]{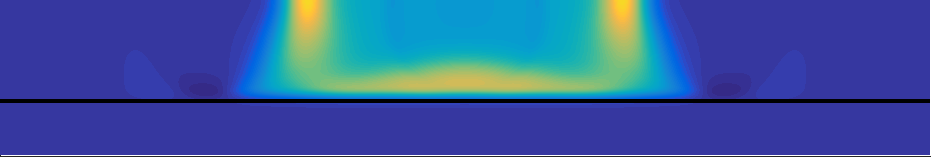}
\\
\includegraphics[width=2.3in]{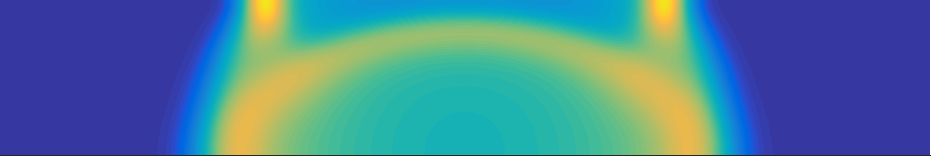}
& 
\includegraphics[width=2.3in]{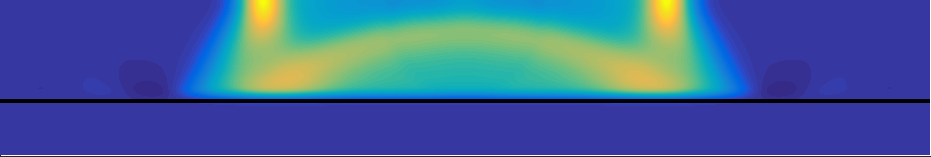}
\\
\includegraphics[width=2.3in]{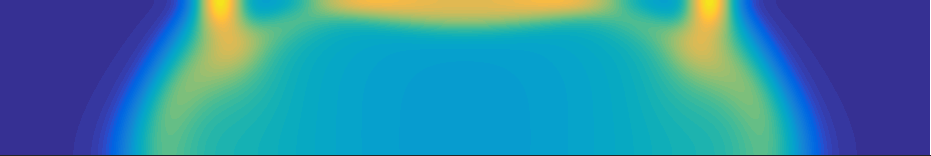}
& 
\includegraphics[width=2.3in]{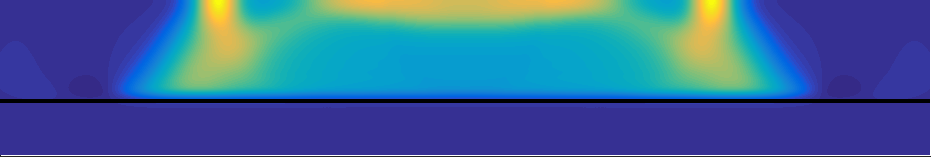}
\\
\includegraphics[width=2.3in]{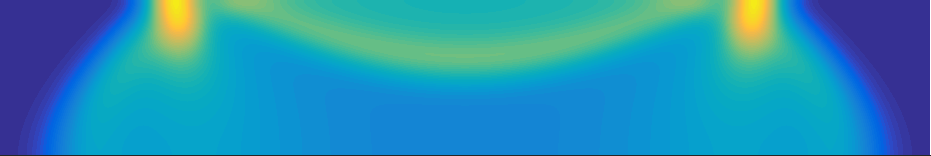}
& 
\includegraphics[width=2.3in]{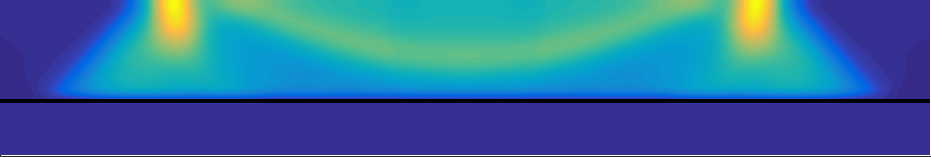}
\\

\bottomrule
\end{tabular}
\caption{True wavefields (left) along with wavefields obtained from
  the moving receivers procedure (right) at times $t = 0.5,
  0.75,\ldots,2.0$. Sources and receivers are placed in $\Gamma =
  [-3.1,3.1] \times \{0\}$, i.e. the top of the images. The known
  region is $M(\Gamma,r)$ with $r = 0.5$. In the snapshots, the known
  region corresponds to the rectangle $[-3.1,3.1] \times [-s, 0]$,
  where $s = e^{1/2} - 1 \approx 0.649$, above the solid black line.}
\label{tbl:ResultImages}
\end{table}

\subsection{Computational implementation of moving sources}
\label{sec:MovingSrcsModification}

To apply the moving sources procedure to a source $F \in L^2([0,T]
\times M(\Gamma,r))$ we need the quantity $\curlyK^* F$. The formula
(\ref{eqn:BlagoTypeInteriorOperatorExpr}) for computing $\curlyK$ uses
the quantity $L$, and as discussed above, it is costly to fully
dicretize $L$. In order to avoid this, we instead compute the action
of $\curlyK^*$ by transpostion. To that end, we note that $\curlyK^* F
= (W^{T/2})^*w^F(T/2,\cdot)$, thus it will suffice to approximate
$(W^{T/2})^*w^F(T/2,\cdot)$.

We first recall from Lemma \ref{lemma:Ltranspose} that $L^* F =
Rw^{RF}|_{[0,T] \times \Gamma}$. Thus, for a basis function $\varphi_i
\in S^T$ we have,
\begin{equation}
  \label{eqn:wFTrace}
  \langle \varphi_i, Rw^F \rangle_{L^2([0,T] \times \Gamma)} = \langle
  L\varphi_i, RF \rangle_{L^2([0,T]\times M(\Gamma,r))}.
\end{equation}
After applying the receiver moving technique to compute $L\varphi_i$,
we can compute the right hand side of this expression. Then,
(\ref{eqn:wFTrace}) allows us to compute the inner-product between
$Rw^F|_{[0,T]\times \Gamma}$ and any basis function, which allows us
to compute the coefficients of the projection of
$Rw^F|_{[0,T]\times\Gamma}$ onto $S^T$. Computing the function
associated with these coefficients and time-reversing the allows us to
approximate $w^F|_{[0,T]\times\Gamma}$.

We now return to the derivation of
(\ref{eqn:BlagoTypeInteriorOperatorExpr}) in order to show how to
approximate $(W^{T/2})^*w^F(T/2,\cdot)$.  Let us suppose that $F \in
C^\infty(M(\Gamma,r)\times[0,T/2])$ and $\varphi_i \in S^{T/2}$. Then,
we define $I_2(t,s) := \langle w^F(t,\cdot),u^{\varphi_i}(s,\cdot)
\rangle_{L^2(M)}$, and observe that $I_2(T/2,T/2) = \langle \varphi_i, (W^{T/2})^*w^F(T/2,\cdot)\rangle$ . We note that $I_2$ is defined analogously to $I$ from
the derivation of (\ref{eqn:BlagoTypeInteriorOperatorExpr}), the only
difference between these expressions is that we have exchanged the
roles of $t$ and $s$. Then, a similar computation to our earlier derivation
shows,
\begin{equation*}
  (\p_t^2 - \p_s^2)I_2(t,s) = \langle w^F(t), \p_\nu u^{\varphi_i}(s) \rangle_{L^2(\p M)}
  - \langle F(t), u^{\varphi_i}(s)\rangle_{L^2(M)}.
\end{equation*}
Applying the definition of $L$, noting that $\p_\nu u^{\varphi_i} =
\varphi_i$, and using the support properties of $\varphi_i$ and $F$ we
can rewrite this as,
\begin{equation*}
  (\p_t^2 - \p_s^2)I_2(t,s) = \langle w^F(t)|_{\Gamma},
  \varphi_i(s) \rangle_{L^2(\Gamma)} - \langle F(t), L\varphi_i(s)\rangle_{L^2(M(\Gamma,r))}.
\end{equation*}
We then use Duhamel's principle and set $t=s=T/2$ in the result to
obtain,
\begin{equation}
  \label{eqn:MvSrcTranspose}
  \begin{split}
    \langle \varphi_i, (W^{T/2})^*w^F(T/2,\cdot)\rangle_{L^2([0,T/2]\times\Gamma)}
    = &~\langle \varphi_i, J^T w^F|_{[0,T] \times
      \Gamma}\rangle_{L^2([0,T/2] \times \Gamma)} \\ & - \langle L\varphi_i,
    J^T F\rangle_{L^2([0,T/2] \times M(\Gamma,r))}.
  \end{split}
\end{equation}
To approximate $J^Tw^F|_{[0,T]\times\Gamma}$, we use the approximation
to $w^F|_{[0,T]\times\Gamma}$ computed from (\ref{eqn:wFTrace}) and
apply the definition of $J^T$. We compute the other term on the right
by directly applying (\ref{eqn:wFTrace}) with $J^T F$ in place of $F$.
Finally, we use the inner-products (\ref{eqn:MvSrcTranspose}) to
compute the coefficients of $(W^{T/2})^*w^F(T/2,\cdot)$ in the basis
for $S^{T/2}$.

We now describe our computational implementation of the moving sources
procedure.  Let us recall that our goal is, for a source $F \in
L^2([0,T/2]\times M(\Gamma,r))$, to approximate the wave $w^F$ in
$M[0,T/2] \times (\Gamma,r)$. By Lemma \ref{thm:MovingSrc}, our first
step in approximating $w^F(t,\cdot)|_{M(\Gamma,r)}$ is to solve a
discrete version of (\ref{eqn:IntSrcControl}). So we solve
(\ref{eqn:discreteControl}) with $\tau = T/2$ and $b = (W^{T/2})^*
w^{Z_{T/2-t} F}(T/2,\cdot)$.  That is, we compute
$h_{\alpha,t}$ by solving
\begin{equation}
  \label{eqn:discreteMvngSrc}
  ([P] [K^{T/2}] [P] + \alpha) [h_{\alpha,t}] = [P] [(W^{T/2})^*
    w^{Z_{T/2-t} F}(T/2,\cdot)],
\end{equation}
where we use (\ref{eqn:MvSrcTranspose}) to compute the right-hand side
of this expression.  Finally, we compute the wave
$u^{h_{\alpha,t}}(T/2,\cdot)$ as in the moving receivers
implementation.

\subsection{Moving sources results}

To demonstrate our moving sources procedure, we consider a source
\begin{equation}
  F(t,x,y) = \exp\left(-a((t - t_c)^2 + (x - x_c)^2 + (y - y_c)^2)\right),
\end{equation} 
where $t_c = 0.1$, $(x_c,y_c) = (0, 0.25)$, and $a = a_t$.  We use the
moving sources procedure to approximate
$w^{F}(t,\cdot)|_{M(\Gamma,r)}$ for several times $t$. That is, for
these $t$ we solve (\ref{eqn:discreteMvngSrc}) using $\alpha =
10^{-4}$ and compute the associated wavefield
$u^{h_{\alpha,t}}(T/2,\cdot)$ approximating $w^F(t,\cdot)$ in
$M(\Gamma,r)$. We compare the results of our procedure with the true
wavefields in Figure \ref{tbl:SrcResultImages}.

\begin{table}[htb!]
\centering
%\begin{tabular}{c{2cm} c{2.2in} c{2.2in}}
\begin{tabular}{c c}
\toprule
True wavefield $w^F(t,\cdot)$ & Approximate wavefield $u^{h_{\alpha,t}}(T,\cdot)$ \\%%in $M(\Gamma,r)$ \\
\midrule
\includegraphics[width=2.3in]{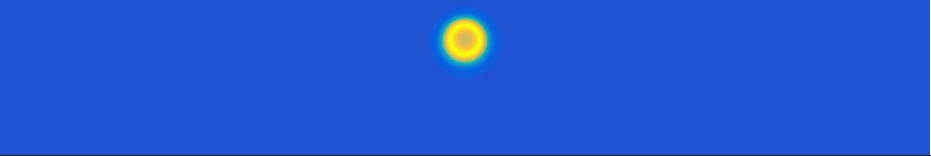}
& 
\includegraphics[width=2.3in]{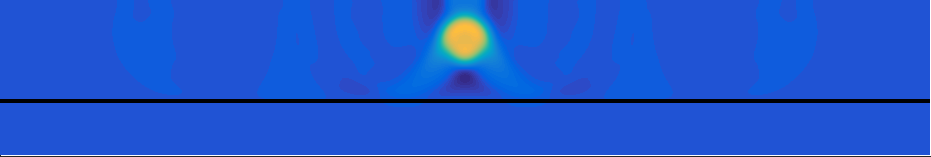}
\\
\includegraphics[width=2.3in]{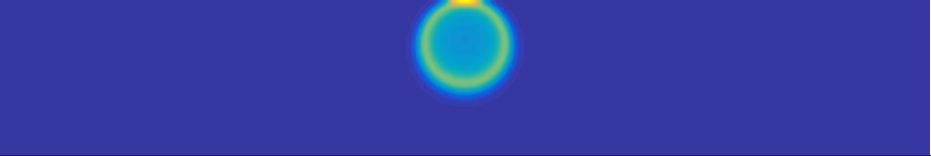}
& 
\includegraphics[width=2.3in]{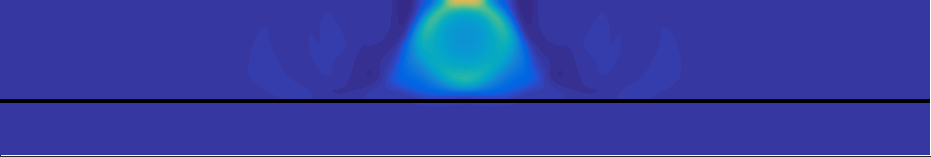}
\\
\includegraphics[width=2.3in]{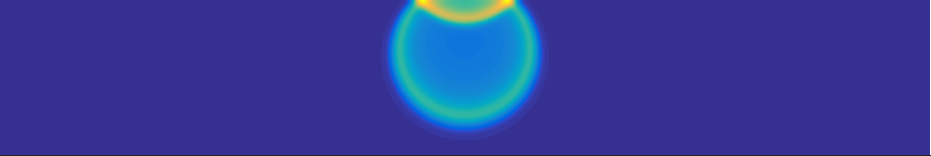}
& 
\includegraphics[width=2.3in]{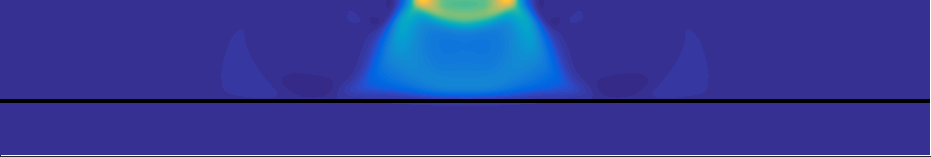}
\\
\includegraphics[width=2.3in]{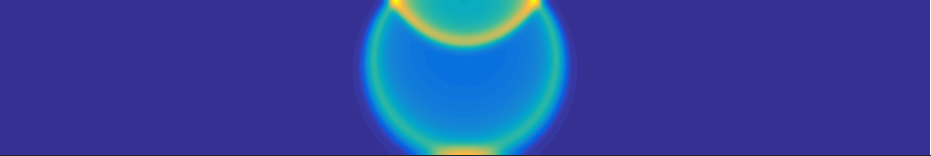}
& 
\includegraphics[width=2.3in]{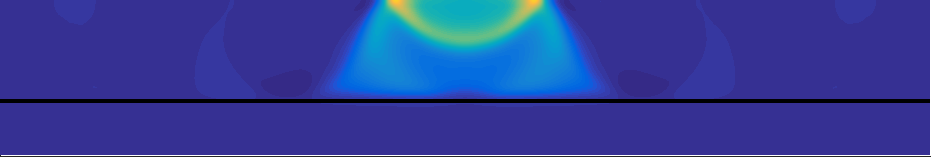}
\\
\includegraphics[width=2.3in]{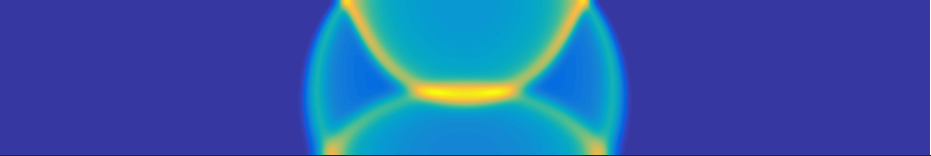}
& 
\includegraphics[width=2.3in]{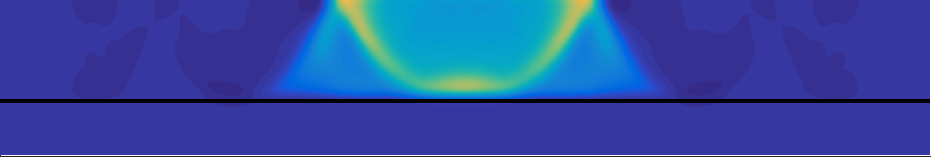}
\\
\includegraphics[width=2.3in]{{./draft_images/MovingSrcsNumerics/Hyperbolic_T_2_Y_1/original/original_0.750000_cropped}.png}
& 
\includegraphics[width=2.3in]{{./draft_images/MovingSrcsNumerics/Hyperbolic_T_2_Y_1/recovered/recovered_0.750000_cropped}.png}
\\
\includegraphics[width=2.3in]{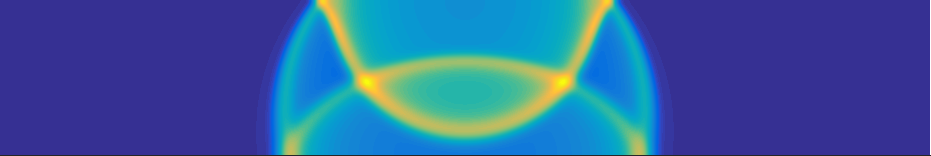}
& 
\includegraphics[width=2.3in]{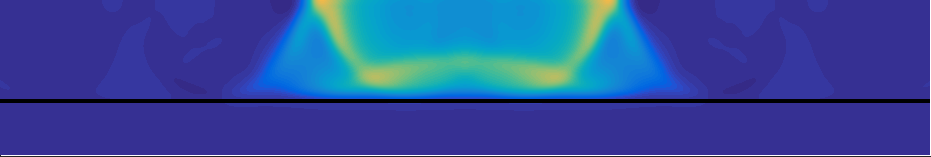}
\\
\includegraphics[width=2.3in]{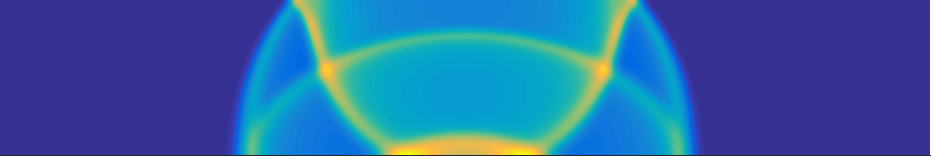}
& 
\includegraphics[width=2.3in]{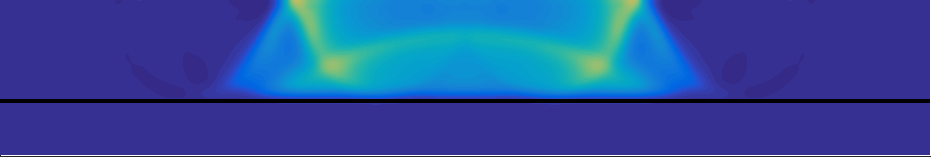}
\\

\bottomrule
\end{tabular}
\caption{We plot the true wavefields (left) along with wavefields
  obtained from the moving sources procedure (right) at times $t =
  0.125, 0.25, \ldots, 1.0$. We again note that, for the moving
  sources wavefields, the sources and receivers are placed in $\Gamma
  = [-3.1,3.1] \times \{0\}$, i.e. the top of the images. The
  known region corresponds to the rectangle $[-3.1,3.1] \times
  [-s,0]$, where $s \approx 0.649$, above the solid black line.}
\label{tbl:SrcResultImages}
\end{table}

\medskip
\noindent{\bf Acknowledgements.}  The authors express their gratitude
to the Institut Henri Poincar\'e where a part of this work has been
done.  The authors thank Jan Boman, Luc Robbiano, J\'{e}r\^{o}me Le
Rousseau and Daniel Tataru for their enlightening discussions.

\bibliographystyle{siamplain}
\bibliography{Bibliography}

\end{document}